\documentclass[reqno,10pt]{amsart}

\usepackage{graphicx}  
\usepackage{epsfig}  
\usepackage{epstopdf}  
\usepackage{pdfpages}  
\usepackage{amssymb}  
\usepackage{empheq}  
\usepackage{cases}  
\usepackage{amsthm,amsmath}  
\usepackage{caption,lipsum}  
\usepackage{stmaryrd}  
\usepackage{tabularx}  
\usepackage{color}  
\usepackage{empheq,float}  
\usepackage{amsfonts,euscript,mathrsfs}  
\usepackage{hyperref}  
\usepackage{booktabs,multirow}  
\usepackage{enumerate} 
\usepackage{enumitem} 
\usepackage{subcaption} 
\usepackage[numbers,sort]{natbib}  
\usepackage[ruled,vlined]{algorithm2e} 

\DeclareGraphicsExtensions{.eps} 
\allowdisplaybreaks[3]  
\numberwithin{equation}{section} 

\newtheorem{theorem}{Theorem}[section]
\newtheorem{lemma}[theorem]{Lemma}

\newtheorem{corollary}[theorem]{Corollary}
\theoremstyle{definition}

\theoremstyle{remark}
\newtheorem{remark}{Remark}
\newtheorem{assumption}{Assumption}

\def\fe{\mbox{e}}

\begin{document}
	
\title[]{Some uniform error analysis for gKdV equations in dispersionless limit regime before dispersive shock} 

\author[B. Li]{Bing Li}
\address{\hspace*{-12pt}B.~Li: School of Mathematical Sciences, Chengdu University of Technology, Chengdu, 610059, China}
\email{libing@cdut.edu.cn}

\author[T. Wang]{Tingfeng Wang}
\address{\hspace*{-12pt}T.~Wang: School of Mathematics and Statistics, Wuhan University, Wuhan, 430072, China}
\email{tingfengwang@whu.edu.cn}

\author[X. Zhao]{Xiaofei Zhao}
\address{\hspace*{-12pt}X.~Zhao: School of Mathematics and Statistics \& Computational Sciences Hubei Key Laboratory, Wuhan University, Wuhan, 430072, China}
\email{matzhxf@whu.edu.cn}

\date{}
\dedicatory{}
\begin{abstract}
This work establishes uniform error estimates for classical numerical schemes applied to the generalized Korteweg-de Vries (gKdV) equation in the dispersionless limit regime, specifically before the development of dispersive shock wave (DSW). We analyze the Crank-Nicolson (CN) and a Lawson-type Runge-Kutta (Lawson-RK) methods, when discretized in space via the Fourier pseudo-spectral method. We prove that both fully discrete schemes achieve optimal second-order temporal accuracy and spectral spatial accuracy, with error constants being uniform in the vanishing dispersion parameter $\varepsilon$. The analysis also addresses the unique solvability of the CN scheme in the dispersionless limit regime. Theoretical findings are supported by numerical experiments, demonstrating the $\varepsilon$-uniform accuracy and the necessity of step size restriction for CN. The study in general validates the classical methods for solving the dispersionless gKdV equation before DSW.
\\
    
\end{abstract}
\maketitle

\section{Introduction}
In this work, we are concerned with the generalized KdV (gKdV) equation \cite{klein2015numerical,bona2013conservative,li2026second,el2016dispersive}
\begin{equation}\label{eq: gKdV}
	\left\{\begin{aligned}
		&\partial_tu(x,t) + \varepsilon^2 \partial_x^3u(x,t)+\frac{1}{k}\partial_x \big(u(x,t)\big)^{k}=0, \quad t>0,
        x\in\mathbb{T},\\
		&u(x,0)=u_0(x),
	\end{aligned}\right.
\end{equation}
where $\varepsilon>0,k\geq2$ are given parameters, and $u(x, t) : \mathbb{T} \times \mathbb{R}^{+} \to \mathbb{R}$ is the unknown solution function imposed for simplicity on the torus $\mathbb{T} = (-\pi, \pi)$ with $u_0$ a given initial data satisfying periodic boundary conditions. The parameter $k=2$  
corresponds to the original KdV equation proposed by Korteweg and de Vries \cite{korteweg1895xli} for the wave dynamics of shallow water surfaces, where the balance between nonlinearity and dispersion leads to the discovery of solitons \cite{Zabusky}, and it is generalized to $k\geq3$ for general nonlinear long-wave dispersive phenomena in a wide range of applications \cite{miura1976korteweg}. In cases where dispersive effects are weak, a small $0<\varepsilon<1$ can be introduced mathematically through a temporal and spatial rescaling, making it a dimensionless parameter denoting the strength of the dispersion of the media. 
With a vanishing $\varepsilon
\ll1$, \eqref{eq: gKdV} is referred to be in the dispersionless limit regime \cite{PNAS,lax1983small} or semi-classical limit regime \cite{masoero2013semiclassical}, and we would in turn refer to the case of a fixed $\varepsilon>0$ as the classical regime. 


In the classical regime of \eqref{eq: gKdV}, we refer the readers to \cite{bao2017global, colliander2003sharp, colliander2004multilinear, killip2019kdv} for its well-posedness theorems.  When $k = 2$ or $3$, the equation \eqref{eq: gKdV} is completely integrable and possesses infinitely many conservation laws \cite{miura1976korteweg}. The integrability is lost when $k \ge 4$, making \eqref{eq: gKdV} only preserve the following three physical quantities: 
\begin{align*}
&\mbox{mass}\quad	M[u] := \int_{\mathbb{T}}{ u(x,t) }\;\text{d}x \equiv M[u_{0}],\quad
\mbox{energy}\quad 
	E[u] := \left\| u(x,t) \right\|_{L^{2}}^2 \equiv E[u_{0}],\\
&\mbox{Hamiltonian}\quad	H[u] := \int_{\mathbb{T}}{ \frac{\varepsilon^2}{2}\left( \partial_{x}u(x,t) \right)^2 - \frac{1}{k(k+1)}\left( u(x,t) \right)^{k+1} }\;\text{d}x \equiv H[u_{0}].
\end{align*}
This loss of complete integrability also means that  the gKdV equation has to be solved numerically in general for wave phenomena of physical interest --- such as finite-time blow-ups \cite{blowupDG,klein2015numerical} and the dispersive shock wave (DSW) \cite{el2016dispersive,klein2015numerical,li2026second}. 
From a numerical perspective, KdV-type equations, given the mathematical and physical significance, have consistently been the subject of computational investigation. For sufficiently smooth solutions, a variety of numerical methods have been proposed and analyzed, including finite difference methods \cite{courtes2020error, klein2008fourth,Zabusky}, operator splitting techniques \cite{holden1999operator, holden2011operator, holden2013operator}, exponential integrators \cite{ostermann2020lawson, klein2015numerical,hofmanova2017exponential}, finite element methods \cite{bona1986fully, bona1995conservative,CNFEM}, spectral methods \cite{dwivedi2026spectral,ma2001optimal, shen2003new}, and discontinuous Galerkin methods \cite{bona2013conservative, liu2006local,DGLiu,yan2002local}. To address computations involving rough solutions, numerous numerical techniques known as low-regularity integrators have also been developed \cite{wu2022optimal, wu2022embedded, li2023gauge, li2026second, rousset2022convergence, li2026unfiltered}. It should be noted, however, that all the aforementioned numerical approaches have been analyzed in the classical regime of \eqref{eq: gKdV}, where the dispersion parameter $\varepsilon = \mathcal{O}(1)$, with no characterization of convergence with respect to the dependence on $\varepsilon$.

In the dispersionless limit regime of \eqref{eq: gKdV}, formally the model will become a Burgers-type equation as $\varepsilon \to 0$, where a smooth initial data can develop shocks in the evolution. The formation of a shock that is simultaneously affected by a small dispersion can trigger complex dynamics, e.g., the DSW. Indeed, this has been intensively studied in the literature \cite{claeys2010painleve,dwivedi2026spectral,lax1983small,grava2002generation,masoero2013semiclassical,grava2012numerical,venakides1990higher} through analytical or numerical approaches. 
For the integrable KdV equation ($k=2$), this singular limit has been studied in depth. In the pioneering work of Lax and Levermore \cite{lax1983small}, the limit was characterized by a constrained variational problem. The local oscillatory structure was explicitly derived by Venakides \cite{venakides1990higher}. The solvability of the Whitham modulation equations for single-phase oscillations after the gradient catastrophe was established by Tian \cite{tian1993oscillations} for monotone or single-hump initial data, and later extended to the multiphase setting by Grava and Tian \cite{grava2002generation}. The universal critical asymptotics near the leading edge of the oscillatory zone, governed by the Painlev\'e II equation, were studied by Claeys and Grava \cite{claeys2010painleve}. Systematic quantitative comparisons between the small-dispersion KdV dynamics and the associated Whitham and Painlev\'e asymptotics were carried out numerically by Grava and Klein \cite{grava2007numerical, grava2012numerical}. For the gKdV equation \eqref{eq: gKdV} with general $k \ge 2$, the dispersionless limit before the gradient catastrophe was rigorously studied by Masoero and Raimondo \cite{masoero2013semiclassical}, who proved convergence to the inviscid Burgers solution in Sobolev spaces and the validity of asymptotic expansions. 
It is understood that there exists a critical time $t_{c}>0$ independent of $\varepsilon$ \cite{lax1983small, masoero2013semiclassical}, such that the solution of \eqref{eq: gKdV} remains smooth, particularly with the spatial-temporal wavelength independent of $\varepsilon$ up to it ($t<t_{c}$). Beyond $t_{c}$, a gradient catastrophe occurs and the DSW follows, where the solution to \eqref{eq: gKdV} is then highly oscillatory in both time and space with frequencies in both time and space inversely dependent on $\varepsilon$. 
Although many numerical simulations have been performed in this regime \cite{CNFEM,bona1986fully,bona1995conservative,bona2013conservative,dwivedi2026spectral,DGLiu,klein2015numerical,grava2007numerical,XiaDG},  
classical schemes such as the Crank-Nicolson (CN) method and the Lawson-type Runge-Kutta (Lawson-RK) methods discussed in this paper lack numerical analysis that tracks the explicit dependence of the convergence on $\varepsilon$. Especially, concerning the smooth dynamics of \eqref{eq: gKdV} before the critical time $t_c$, 
classical discretizations may  empirically still retain their effectiveness.


In this paper, we partially address this gap by establishing the uniform convergence of the CN scheme and a Lawson-RK scheme applied to the gKdV equation \eqref{eq: gKdV} to some $T\in(0,t_{c})$ at which no DSW has yet formed and all temporal-spatial derivatives of $u$ stay uniformly bounded for $\varepsilon\in(0,1]$. More precisely, the main results of our study include the following.
\begin{enumerate}[label=(\roman*)]
	\item The two time integrators concerned are discretized with
    the Fourier pseudo-spectral method in space, resulting in two fully discrete schemes for the gKdV equation \eqref{eq: gKdV}. The CN scheme, applied directly to the original variable $u$, is implicit but preserves the discrete mass and the discrete energy. We establish its existence and uniqueness of the solution for all $\varepsilon\in(0,1]$ via a fixed-point iteration converging with a contraction constant independent of $\varepsilon$. 
    The Lawson-RK scheme, applying RK2 to the twisted variable $v = \fe^{t\varepsilon^{2}\partial_{x}^{3}}u$, is explicit and preserves the discrete mass. \vspace{0.5ex}
	
	\item We prove that both schemes achieve optimally and uniformly second-order convergence in time and spectral accuracy in space. More precisely, the error satisfies $\|u(t_n)-u_N^n\|_{H^\gamma}\le C(\tau^2+N^{-m+1})$, where the constant $C$ is independent of $\varepsilon$, $\tau$, and $N$. This uniformity in $\varepsilon$ ensures that the convergence rates remain valid as $\varepsilon\to0$, which is essential for reliable simulations in the dispersionless limit regime. To our knowledge, this constitutes the first rigorous error analysis of fully discrete schemes for gKdV equations in this regime. \vspace{0.5ex}
	
\end{enumerate}
The theoretical results are validated by numerical experiments. 
Although restricted, our study indeed forms a necessary step before sophisticated multiscale methods for improvements could kick in.  

The remainder of this paper is organized as follows.  Section~\ref{sec:2} presents the two numerical schemes.  Section~\ref{sec:3} gives the existence, uniqueness, conservation laws, and convergence analysis of the CN scheme.  Section~\ref{sec:4} is devoted to the convergence analysis of the RK2 scheme. Numerical experiments are reported in Section~\ref{sec:5} and some conclusions are drawn in Section~\ref{sec:con}. Throughout the paper, we adopt the following notational convention. For any two quantities $A$ and $B$, the notation $A \lesssim B$ signifies the existence of a constant $C > 0$ independent of $\varepsilon$, the time step $\tau$ and the number of spatial grid points $N$, such that $|A| \le C B$.  The value of the constant $C$ may change from line to line.

\section{The considered numerical schemes}\label{sec:2}

In this section, we present the two numerical schemes considered in this paper. We first show the discretizations in time, yielding generic semi-discrete schemes that can, in principle, be combined with any consistent spatial discretizations. Then, we specialize to the Fourier pseudo-spectral method for the spatial discretization, arriving at the fully discrete schemes for convergence analysis.

\subsection{CN and Lawson-RK integrators}

Let $\tau>0$ denote the temporal step and set $t_{n} = n\tau$ as discrete time grid points. We denote by $u^{n} \approx u(\cdot, t_n)$ the semi-discrete approximation, which remains a function of the spatial variable $x \in \mathbb{T}$.  The two classical time-stepping schemes are given below.

The CN method is one of the most classical and widely used temporal discretizations for evolutionary PDEs, owing to its second-order accuracy, unconditional linear stability, and structure-preserving properties. Applied to the gKdV equation \eqref{eq: gKdV}, the time semi-discrete \textbf{CN} scheme reads
\begin{align}
	\frac{1}{\tau}\bigl( u^{n+1} - u^{n} \bigr) + \varepsilon^{2} \partial_{x}^{3} u^{n+\frac{1}{2}} + \frac{1}{k} \partial_{x} f\bigl( u^{n+\frac{1}{2}} \bigr) = 0, \qquad x \in \mathbb{T},\; n = 0, 1, \dots, \label{scm: CN-semi}
\end{align}
where $u^{n+\frac{1}{2}} := \frac{1}{2}( u^{n+1} + u^{n} )$ denotes the time average and $f(z) := z^{k}$. 
Combined with various spatial discretizations, \eqref{scm: CN-semi} forms a natural choice to solve \eqref{eq: gKdV} but has only been analyzed numerically so far in the case of $\varepsilon=1$, e.g., its convergence with the finite element method \cite{CNFEM} on the classical KdV equation. 

At each time level, \eqref{scm: CN-semi} constitutes a nonlinear equation to solve for $u^{n+1}$, necessitating an iterative solver which may become computationally demanding. An explicit scheme is certainly demanding to avoid this issue, and here we consider the alternative approach based on the \emph{twisted variable}. The third-order dispersive term $\varepsilon^{2} \partial_{x}^{3} u$ is very stiff and, if treated explicitly, can lead to a strong numerical stability problem. To eliminate this stiffness and construct an explicit scheme, we work on the twisted variable $v := \fe^{t\varepsilon^{2}\partial_{x}^{3}} u$ that absorbs the dispersive term into the unknown, thus transforming the gKdV equation \eqref{eq: gKdV} into a form involving only lower order derivatives:  
\begin{align}
	\partial_{t} v(x, t) = -F(t, v(x, t)), \quad t>0,\; x\in\mathbb{T}, \label{eq: gKdV for v}
\end{align}
where we define $F(t, v) := \frac{1}{k} \fe^{t\varepsilon^{2}\partial_{x}^{3}} \partial_{x} \bigl( \fe^{-t\varepsilon^{2}\partial_{x}^{3}} v \bigr)^{k}$ for brevity.
Let $v^{n} \approx v(\cdot, t_n)$ denote the semi-discrete approximation. Applying the classical second order Runge-Kutta method to \eqref{eq: gKdV for v} gives:
\begin{subequations}\label{scm: RK2-semi}
	\begin{align}
		v_{1}^{n+1} &= v^{n} - \tau F(t^{n}, v^{n}), \label{scm: RK2-semi a}\\
		v^{n+1} &= v^{n} - \frac{\tau}{2} 
		\bigl[ F(t^{n}, v^{n}) + F(t^{n+1}, v_{1}^{n+1}) \bigr], \label{scm: RK2-semi b}
	\end{align}
\end{subequations}
and the semi-discrete solution in the original variable can be recovered via $u^{n} = \fe^{-t_n\varepsilon^{2}\partial_{x}^{3}} v^{n}$. 
Such a scheme to solve \eqref{eq: gKdV} belongs to the Lawson-type method \cite{HochOst}, and we will refer to it as the \textbf{Lawson-RK} scheme.

\subsection{Full discretization}
The schemes \eqref{scm: CN-semi} and \eqref{scm: RK2-semi} can be implemented under various spatial discretizations in practice. Here to be precise for up-coming analysis, we focus on the Fourier pseudo-spectral discretization \cite{shen2011spectral} with the following notations. For any $u(x)\in L^{2}(\mathbb{T})$, its Fourier expansion reads
\begin{align*}
	u(x) = \sum_{\xi\in \mathbb{Z}} \widehat{u}(\xi) \, \fe^{\mathrm i\xi (x+\pi)}, \quad 
	\text{with} \quad \widehat{u}(\xi) = \frac{1}{2\pi}\int_{\mathbb{T}} u(x)\fe^{-\mathrm i\xi (x+\pi)}\,\mathrm{d}x.
\end{align*}
Let $N$ be an even positive integer  and denote the orthogonal projection operator $P_{N} : L^{2}(\mathbb{T}) \to Y_{N}$ as 
\begin{align*}
	(P_{N}u)(x) := \sum_{\xi=-N/2}^{N/2} \widehat{u}(\xi) \, \fe^{\mathrm i\xi (x+\pi)},\quad 
    Y_{N} := \operatorname{span}\left\{ \fe^{\mathrm i\xi (x+\pi)} : \xi = -\frac{N}{2}, -\frac{N}{2}+1, \ldots, \frac{N}{2} \right\}.
\end{align*}
The projection $P_N$ commutes with $\partial_x$, i.e., $P_N\partial_x = \partial_x P_N$. Set the spatial mesh size $h = \frac{2\pi}{N}$ and the grid points are given by $x_j = -\pi + jh$ for $j = 0,1,\dots, N$, covering the periodic domain $[-\pi,\pi]$ with $x_0 \equiv x_N$. The trigonometric interpolation operator $I_{N}: L^{2}(\mathbb{T}) \to Y_{N}$ is then denoted as
\begin{align*}
	(I_{N}u)(x) = \sum_{\xi=-N/2}^{N/2-1} \widetilde{u}(\xi) \fe^{\mathrm i\xi (x+\pi)}, \qquad
	(I_{N}u)(x_{j}) = u(x_{j}), \quad j = 0, 1 ,\dots, N-1,
\end{align*}
where the discrete Fourier coefficients $\{\widetilde{u}(\xi)\}$ are given via the discrete Fourier transform:
\begin{align*}
	\widetilde{u}(\xi) = \frac{1}{N} \sum_{j=0}^{N-1} u(x_{j}) \fe^{-\mathrm i\xi (x_{j}+\pi)}, \qquad 
	\xi = -\frac{N}{2}, -\frac{N}{2}+1, \ldots, \frac{N}{2}-1.
\end{align*}
Compared with $P_N$, whose frequency range extends to $\xi = N/2$, the interpolation operator $I_N$ omits the highest mode $\xi = N/2$, which can be practically implemented via the fast Fourier transform. 
With these preparations, we now state the fully discrete approximation $u_{N}^{n} \in Y_{N}$ for $u(\cdot, t_n)$. 

Projecting the semi-discretization of CN \eqref{scm: CN-semi} onto $Y_{N}$ yields the fully discrete \textbf{CN} scheme:
\begin{align}
	\frac{1}{\tau}\bigl( u_{N}^{n+1} - u_{N}^{n} \bigr) 
	+ \varepsilon^{2} \partial_{x}^{3} u_{N}^{n+\frac{1}{2}} 
	+ \frac{1}{k} P_{N} \partial_{x} f\bigl( u_{N}^{n+\frac{1}{2}} \bigr) = 0, \label{scm: CN-FP}
\end{align}
where $u_{N}^{n+\frac{1}{2}} := \frac{1}{2}( u_{N}^{n+1} + u_{N}^{n} )$. The projection $P_N$ in front of the nonlinear term guarantees that the equation is well defined on the finite-dimensional space $Y_{N}$.  Clearly, \eqref{scm: CN-semi} or \eqref{scm: CN-FP} is symmetric in time and, as will be proved in Lemma~\ref{lem: conservative law}, the fully discrete CN scheme \eqref{scm: CN-FP} preserves both the discrete mass and the discrete energy exactly. 
To solve the nonlinear scheme  \eqref{scm: CN-FP}, we present an iterative solver below. Rewriting \eqref{scm: CN-FP} as
\begin{align*}
	\Bigl( I + \frac{\varepsilon^{2}\tau}{2} \partial_{x}^{3} \Bigr) u_{N}^{n+1} 
	= \Bigl( I - \frac{\varepsilon^{2}\tau}{2} \partial_{x}^{3} \Bigr) u_{N}^{n} 
	- \frac{\tau}{k} \partial_{x} P_{N} f\Bigl( \frac{u_{N}^{n+1} + u_{N}^{n}}{2} \Bigr),
\end{align*}
and introducing the operators $\mathcal{A}_{\tau,\varepsilon}^{\pm} = I \pm \frac{1}{2} \varepsilon^{2} \tau \partial_{x}^{3}$, we obtain the fixed-point formulation
\begin{align*}
	u_{N}^{n+1} = \bigl( \mathcal{A}_{\tau,\varepsilon}^{+} \bigr)^{-1} \mathcal{A}_{\tau,\varepsilon}^{-} \, u_{N}^{n} 
	- \frac{\tau}{k} \bigl( \mathcal{A}_{\tau,\varepsilon}^{+} \bigr)^{-1} \partial_{x} P_{N} 
	f\Bigl( \frac{u_{N}^{n+1} + u_{N}^{n}}{2} \Bigr).
\end{align*}
Define the mapping $\psi_{w} \colon Y_{N} \to Y_{N}$, depending on $w\in Y_{N}$, by
\begin{align*}
	\psi_{w}(\phi) = \bigl( \mathcal{A}_{\tau,\varepsilon}^{+} \bigr)^{-1} \mathcal{A}_{\tau,\varepsilon}^{-} \, w 
	- \frac{\tau}{k} \bigl( \mathcal{A}_{\tau,\varepsilon}^{+} \bigr)^{-1} \partial_{x} P_{N} 
	f\Bigl( \frac{w + \phi}{2} \Bigr).
\end{align*}
Then, given $u_{N}^{n}$, the solution $u_{N}^{n+1}$ is obtained via the iteration
\begin{align}
	u_{N,(i+1)}^{n+1} = \psi_{u_{N}^{n}}\bigl( u_{N,(i)}^{n+1} \bigr), \qquad i = 0, 1, \dots \label{eq: fixed-point iteration}
\end{align}
with initial guess $u_{N,(0)}^{n+1} = u_{N}^{n}$. Under restriction of $\tau$, the convergence of this iteration that is uniform in $\varepsilon$ will be  established in Corollary~\ref{crlr:iteration-convergence}.

On the other hand, projecting the semi-discretization of Lawson-RK \eqref{scm: RK2-semi} onto $Y_{N}$ with $v^{n}$ replaced by $v_{N}^{n} \in Y_{N}$, we obtain in the fully discrete \textbf{Lawson-RK} scheme:
\begin{subequations}\label{scm: RK2-FP}
	\begin{align}
		v_{N, 1}^{n+1} &= v_{N}^{n} - \tau P_{N} F(t^{n}, v_{N}^{n}), \label{scm: RK2-FP a}\\
		v_{N}^{n+1} &= v_{N}^{n} - \frac{\tau}{2} P_{N} 
		\bigl[ F(t^{n}, v_{N}^{n}) + F(t^{n+1}, v_{N, 1}^{n+1}) \bigr]. \label{scm: RK2-FP b}
	\end{align}
\end{subequations}
This scheme is fully explicit: each time step requires only two evaluations of the nonlinearity $P_{N} F$.  It preserves the discrete mass $M[u_N^n] \equiv M[u_N^0]$ for all $n$, as can be verified directly by summation of the numerical solution over the spatial domain. The fully discrete approximation in the original solution $u$ of \eqref{eq: gKdV} is recovered via $u_{N}^{n} = \fe^{-t_n\varepsilon^{2}\partial_{x}^{3}} v_{N}^{n}$.

For both schemes, the initial datum is set as $u_{N}^{0} = P_{N} I_{2N} \, u_{0}$ for rigorous analysis purposes. 
\begin{remark}
Our use of the Fourier pseudo-spectral method differs somewhat from the conventional practice. In standard implementations, it is common to employ the interpolation operator $I_N$ to evaluate nonlinear terms. 
    In contrast, our schemes rely exclusively on the $P_N$ projection, and all nonlinear terms are computed via convolutions of Fourier coefficients in frequency space. 
    Admittedly, the use of $P_N$ incurs additional computational costs compared to the FFT-based $I_N$ approach. Nevertheless,
 there are two main reasons for this design. 
	First, the $P_N$ projection commutes with the differential operator $\partial_x$, whereas $I_N$ does not.  This commutation property is indispensable in the subsequent error analysis: when estimating the nonlinear terms, the failure of $I_N$ to commute with $\partial_x$ would introduce aliasing errors that are fundamentally difficult to control at the level of precision required for our convergence estimates.  
	Second, our $P_N$ is defined with a frequency-symmetric truncation $\xi = -N/2, \dots, N/2$, which preserves the conjugate symmetry of the Fourier coefficients, i.e., $\widehat{u}(-\xi) = \overline{\widehat{u}(\xi)}$.  This property guarantees that the numerical solution remains strictly real-valued throughout the time evolution. By contrast, $I_N$, which omits the highest mode $\xi = N/2$, gives rise to spurious imaginary components that would compromise the physical fidelity of the approximation.
\end{remark}

\section{Analysis of the CN scheme}\label{sec:3}

In this section, we first collect several auxiliary lemmas, and then establish two sets of results for the fully discrete CN scheme~\eqref{scm: CN-FP}:
(i)~existence and uniqueness of the numerical solution at each time step as well as the discrete conservation laws; 
(ii)~a rigorous convergence analysis with optimal error estimates. 

\subsection{Notation and technical lemmas}
\label{sec:3.1}

Let $\langle u, v \rangle := \int_{\mathbb{T}} u(x)\,\overline{v(x)} \,\mathrm{d}x$ be the standard $L^{2}$ inner product on $\mathbb{T}$, 
where $\overline{v}$ denotes the complex conjugate of $v$. The conjugate in the inner product can be safely dropped when the argument is real-valued function. This simplification will be used in some analysis without further comment. For any $\gamma \ge 0$, the Sobolev space $H^{\gamma}(\mathbb{T})$ is equipped with the equivalent inner product and norm
\begin{align*}
	\langle u,v\rangle_{H^{\gamma}} = \langle J^{\gamma} u, J^{\gamma} v\rangle, \qquad
	\|u\|_{H^{\gamma}} = \sqrt{\langle J^{\gamma} u, J^{\gamma} u\rangle}
	= \Bigg(\sum_{\xi\in\mathbb{Z}} (1+\xi^2)^{\gamma}\, |\widehat{u}(\xi)|^2\Bigg)^{\frac12},
\end{align*}
where $J^{\gamma} = (1-\partial_{xx})^{\gamma/2}$. We begin with two classical inequalities for pseudo-spectral approximations.

\begin{lemma}[Bernstein-type inequalities \cite{shen2011spectral}]\label{lem:operator-est}
	For any $\gamma>0$ and $m>0$, if $u$ is periodic function and $u \in H^{\gamma + m}(\mathbb{T})$, then
	\begin{align*}
		\| (P_{N} - I)u \|_{H^{\gamma}} &\le N^{-m} \| u \|_{H^{\gamma + m}}, \qquad
		\| P_{N}u \|_{H^{\gamma}} \le \| u \|_{H^{\gamma}},
	\end{align*}
	and, if in addition $\gamma + m > \frac{1}{2}$, then
	\begin{align*}
		\| I_{N}u - u \|_{H^{\gamma}} &\le N^{-m} \| u \|_{H^{\gamma + m}}, \qquad
		\| I_{N}u \|_{H^{\gamma + m}} \le C \| u \|_{H^{\gamma+ m}}.
	\end{align*}
\end{lemma}

\begin{lemma}[Kato-Ponce-type commutator estimate]\label{lem:Kato-Ponce}
	For any $\gamma > \frac{3}{2}$ and $u, v \in H^{\gamma}(\mathbb{T})$ are both real-valued, the following commutator estimate holds:
	\begin{align}
		\bigl\| \left[ J^{\gamma}, v \right] \partial_{x} u \bigr\|_{L^{2}}
		\lesssim \|u\|_{H^{\gamma}} \|v\|_{H^{\gamma}}, \label{eq:commutator}
	\end{align}
	where $\left[ J^{\gamma}, v \right] \partial_{x} u = J^{\gamma}(\partial_{x} u \cdot v) - J^{\gamma} \partial_{x} u \cdot v$, Consequently,
	\begin{align}
		\big\langle u, \, \partial_{x}u \cdot v \big\rangle_{H^{\gamma}}
		\lesssim \| u \|_{H^{\gamma}}^{2} \, \| v \|_{H^{\gamma}}. \label{eq:Kato-Ponce-result}
	\end{align}
\end{lemma}
\begin{proof}
    The proof of \eqref{eq:commutator} is given in Lemma 3.2 and Remark 3.3 of \cite{wu2022optimal}; it remains only to prove \eqref{eq:Kato-Ponce-result}. Expanding the $H^{\gamma}$ inner product yields
	\begin{align*}
		\big\langle u, \partial_{x} u \cdot v \big\rangle_{H^{\gamma}}
		= \big\langle J^{\gamma} u, J^{\gamma}(\partial_{x} u \cdot v) \big\rangle
		= \big\langle J^{\gamma} u, J^{\gamma} \partial_{x} u \cdot v \big\rangle
		+ \big\langle J^{\gamma} u, \left[ J^{\gamma}, v \right] \partial_{x} u \big\rangle.
	\end{align*}
	For the first term, we integrate by parts:
	\begin{align*}
		\big\langle J^{\gamma} u, J^{\gamma} \partial_{x} u \cdot v \big\rangle
		= \frac{1}{2} \int_{\mathbb{T}} \partial_{x}\bigl( J^{\gamma} u \bigr)^{2} \cdot v \,\mathrm{d}x
		= -\frac{1}{2} \int_{\mathbb{T}} \bigl( J^{\gamma} u \bigr)^{2} \cdot \partial_{x} v \,\mathrm{d}x.
	\end{align*}
	By the Sobolev embedding $H^{\gamma_{1}}(\mathbb{T}) \hookrightarrow L^{\infty}(\mathbb{T})$ 
	for any $\gamma_{1} > \frac{1}{2}$, and using duality,
	\begin{align*}
		\Bigl| \int_{\mathbb{T}} \bigl( J^{\gamma} u \bigr)^{2} \cdot \partial_{x} v \,\mathrm{d}x \Bigr|
		\lesssim \bigl\| (J^{\gamma} u)^{2} \bigr\|_{L^{1}} \| \partial_{x} v \|_{L^{\infty}}
		\lesssim \| u \|_{H^{\gamma}}^{2} \, \| v \|_{H^{1+\gamma_{1}}}
		\lesssim \| u \|_{H^{\gamma}}^{2} \, \| v \|_{H^{\gamma}},
	\end{align*}
	where the last inequality uses $1+\gamma_{1} \le \gamma$, which holds for a suitable choice of $\gamma_{1} \in (\frac{1}{2}, \gamma-1]$ under the assumption $\gamma > \frac{3}{2}$. 
	
	For the second term, applying the Cauchy--Schwarz inequality together with 
	the commutator estimate~\eqref{eq:commutator} yields
	\begin{align*}
		\big\langle J^{\gamma} u, \left[ J^{\gamma}, v \right] \partial_{x} u \big\rangle
		\le \| J^{\gamma} u \|_{L^{2}} \,
		\bigl\| \left[ J^{\gamma}, v \right] \partial_{x} u \bigr\|_{L^{2}}
		\lesssim \| u \|_{H^{\gamma}}^{2} \, \| v \|_{H^{\gamma}}.
	\end{align*}
	Combining the estimates for both terms gives~\eqref{eq:Kato-Ponce-result}.
\end{proof}

To handle the nonlinearity $f(z)=z^{k}$ compactly, we introduce the following
auxiliary polynomials.  For all $u,v \in \mathbb{R}(x)$, define
\begin{align*}
	&Q_{1}(u, v) = \partial_{x} u \, \sum_{j=0}^{k-2} u^{j} v^{k-2-j}, \quad Q_{2}(u) = u^{k-1}. 
\end{align*}
The direct calculation yields
\begin{align}
	\frac{1}{k}\,\partial_{x}\bigl[ f(u) - f(v) \bigr]
	&= \partial_{x}u \cdot u^{k-1} - \partial_{x}v \cdot v^{k-1} = \partial_{x}u\,(u^{k-1}-v^{k-1}) + \partial_{x}(u-v)\,v^{k-1} \nonumber\\
	&= (u-v)\,\partial_{x}u\sum_{j=0}^{k-2} u^{j} v^{k-2-j}
	+ \partial_{x}(u-v)\,v^{k-1}
	= (u-v) Q_{1}(u, v) + \partial_{x}(u-v)\,Q_{2}(v). \label{eq:f-f-3}
\end{align}

\subsection{Existence, uniqueness, and conservation laws}
\label{sec: 3.2}

For a given $\gamma > 0$ and $M > 0$, define the closed ball $B_{M} := \bigl\{ \phi \in Y_{N} : \| \phi \|_{H^{\gamma}} \le M \bigr\}$.

\begin{lemma}[Contraction property and well-posedness]\label{lem:contraction-map}
	Assume that $\gamma > \frac{1}{2}$ and  
	\begin{align}
		\tau \lesssim \min\bigl\{ \varepsilon^{1+\sigma_{0}},\; N^{-(1+\sigma_{0})} \bigr\}, \label{eq:tau-CN}
	\end{align}
    for any fixed $\sigma_{0}>0$ that can be  arbitrarily small. 
	Then, the following statements hold.
	
	\begin{enumerate}[label=(\roman*)]
		\item For any $\phi_{1},\phi_{2}\in Y_{N}$, there exist $\sigma>0$ (depending on
		$\sigma_{0}$) and a constant $C\bigl(\|w\|_{H^{\gamma}},\|\phi_{1}\|_{H^{\gamma}},
		\|\phi_{2}\|_{H^{\gamma}}\bigr)$ such that
		\begin{align}
			\bigl\| \psi_{w}(\phi_{1}) - \psi_{w}(\phi_{2}) \bigr\|_{H^{\gamma}}
			\le C\bigl(\|w\|_{H^{\gamma}},\|\phi_{1}\|_{H^{\gamma}},\|\phi_{2}\|_{H^{\gamma}}\bigr)
			\, \tau^{\sigma} \, \|\phi_{1}-\phi_{2}\|_{H^{\gamma}}. \label{eq:contraction}
		\end{align}
		
		\item Given any $M>0$ with $\|w\|_{H^{\gamma}}\le M$, there exists a $\tau_{M}>0$ such
		that for all $\tau<\tau_{M}$, the mapping $\psi_{w}$ admits a unique fixed
		point $\phi^{*}\in B_{2M}$ satisfying $\phi^{*}=\psi_{w}(\phi^{*})$.
	\end{enumerate}
\end{lemma}
\begin{proof}
	\textbf{(i)}  From the definition of $\psi_{w}$,
	\begin{align}
		\psi_{w}(\phi_{1}) - \psi_{w}(\phi_{2})
		= -\frac{\tau}{k} \bigl( \mathcal{A}_{\tau,\varepsilon}^{+} \bigr)^{-1}
		\partial_{x} P_{N}
		\Bigl[ f\Bigl( \frac{w+\phi_{1}}{2} \Bigr)
		- f\Bigl( \frac{w+\phi_{2}}{2} \Bigr) \Bigr]. \label{eq:psi-diff}
	\end{align}
	In the frequency domain, for any $\phi\in H^{\gamma}(\mathbb{T})$,
	\begin{align*}
		\bigl\| \tau \bigl( \mathcal{A}_{\tau,\varepsilon}^{+} \bigr)^{-1}
		\partial_{x} P_{N}\,\phi \bigr\|_{H^{\gamma}}^{2}
		= \sum_{\xi=-N/2}^{N/2} (1+\xi^{2})^{\gamma}
		\Bigl| \frac{\mathrm i \tau \xi}{1 - \frac{1}{2}\mathrm i\,\varepsilon^{2}\tau\xi^{3}} \Bigr|^{2}
		\, \bigl| \widehat{\phi}(\xi) \bigr|^{2}.
	\end{align*}
	The case $\xi=0$ is trivial.  For $\xi\neq 0$, optimizing the multiplier yields
	\begin{align*}
		\Bigl| \frac{\mathrm i \tau \xi}{1 - \frac{1}{2}\mathrm i\,\varepsilon^{2}\tau\xi^{3}} \Bigr|^{2}
		\lesssim \max\Bigl\{ \Bigl( \frac{\tau}{\varepsilon} \Bigr)^{\frac{4}{3}},\; \tau^{2} N^{2} \Bigr\},
	\end{align*}
	and consequently
	\begin{align}
		\bigl\| \tau \bigl( \mathcal{A}_{\tau,\varepsilon}^{+} \bigr)^{-1}
		\partial_{x} P_{N}\,\phi \bigr\|_{H^{\gamma}}^{2}
		\lesssim \max\Bigl\{ \Bigl( \frac{\tau}{\varepsilon} \Bigr)^{\frac{4}{3}},\;
		\tau^{2} N^{2} \Bigr\} \, \| \phi \|_{H^{\gamma}}^{2}. \label{eq:Ainv-bound}
	\end{align}
	Using the two bounds encoded in the step-size condition~\eqref{eq:tau-CN},
	\begin{align*}
		\Bigl( \frac{\tau}{\varepsilon} \Bigr)^{\frac{4}{3}} = \tau^{\frac{4\sigma_{0}}{3(1+\sigma_{0})}}
		\cdot \tau^{\frac{4}{3(1+\sigma_{0})}} \varepsilon^{-\frac{4}{3}}
		\lesssim \tau^{\frac{4\sigma_{0}}{3(1+\sigma_{0})}}, \quad
		\tau^{2} N^{2}
		= \tau^{\frac{2\sigma_{0}}{1+\sigma_{0}}}
		\cdot \tau^{\frac{2}{1+\sigma_{0}}} N^{2}
		\lesssim \tau^{\frac{2\sigma_{0}}{1+\sigma_{0}}}.
	\end{align*}
	Hence,
	\begin{align}
		\bigl\| \tau \bigl( \mathcal{A}_{\tau,\varepsilon}^{+} \bigr)^{-1}
		\partial_{x} P_{N}\,\phi \bigr\|_{H^{\gamma}}^{2}
		\lesssim \tau^{2\sigma} \| \phi \|_{H^{\gamma}}^{2}, \qquad
		\sigma := \begin{cases}
			\displaystyle \frac{2\sigma_{0}}{3(1+\sigma_{0})}, & \tau \le 1, \\[10pt]
			\displaystyle \frac{\sigma_{0}}{1+\sigma_{0}},      & \tau > 1.
		\end{cases} \label{eq:Ainv-sigma}
	\end{align}
	Moreover, since $\gamma > \frac{1}{2}$, the Sobolev space $H^{\gamma}$ is a
	Banach algebra, and therefore
	\begin{align}
		\Bigl\| f\Bigl( \frac{w+\phi_{1}}{2} \Bigr)
		- f\Bigl( \frac{w+\phi_{2}}{2} \Bigr) \Bigr\|_{H^{\gamma}}
		&= \Bigl\| \frac{\phi_{1}-\phi_{2}}{2} 
		\sum_{j=0}^{k-1}
		\Bigl( \frac{w+\phi_{1}}{2} \Bigr)^{k-1-j}
		\Bigl( \frac{w+\phi_{2}}{2} \Bigr)^{j} \Bigr\|_{H^{\gamma}} \notag \\
		&\le C\bigl( \|w\|_{H^{\gamma}}, \|\phi_{1}\|_{H^{\gamma}},
		\|\phi_{2}\|_{H^{\gamma}} \bigr) \,
		\|\phi_{1}-\phi_{2}\|_{H^{\gamma}}. \label{eq:f-diff-CN}
	\end{align}
	Substituting~\eqref{eq:Ainv-sigma} and~\eqref{eq:f-diff-CN}
	into~\eqref{eq:psi-diff} gives~\eqref{eq:contraction}, which proves~\textit{(i)}.
	
	\textbf{(ii)}  Restricted to $B_{2M}$, the estimate~\eqref{eq:contraction}
	provides a constant $C_{1,M}>0$ (depending on $M$) such that
	\begin{align*}
		\| \psi_{w}(\phi_{1}) - \psi_{w}(\phi_{2}) \|_{H^{\gamma}}
		\le C_{1,M}\,\tau^{\sigma} \, \|\phi_{1}-\phi_{2}\|_{H^{\gamma}},
		\qquad \forall\, \phi_{1},\phi_{2}\in B_{2M}.
	\end{align*}
	In particular, for any $\phi\in B_{2M}$,
	\begin{align*}
		\| \psi_{w}(\phi) - \psi_{w}(0) \|_{H^{\gamma}}
		\le C_{1,M}\,\tau^{\sigma} \|\phi\|_{H^{\gamma}}
		\le 2\,C_{1,M}\,\tau^{\sigma} M.
	\end{align*}
	To estimate $\|\psi_{w}(0)\|_{H^{\gamma}}$, observe that
	\begin{align*}
		\bigl\| \bigl( \mathcal{A}_{\tau,\varepsilon}^{+} \bigr)^{-1}
		\mathcal{A}_{\tau,\varepsilon}^{-} w \bigr\|_{H^{\gamma}}^{2}
		= \sum_{\xi=-N/2}^{N/2} (1+\xi^{2})^{\gamma}
		\Bigl| \frac{1+\frac{1}{2}\mathrm i\,\varepsilon^{2}\tau\xi^{3}}
		{1-\frac{1}{2}\mathrm i\,\varepsilon^{2}\tau\xi^{3}} \Bigr|^{2}
		\, |\widehat{w}(\xi)|^{2}
		= \| w \|_{H^{\gamma}}^{2}.
	\end{align*}
	Together with~\eqref{eq:Ainv-sigma},
	\begin{align*}
		\| \psi_{w}(0) \|_{H^{\gamma}}
		&\le \bigl\| \bigl( \mathcal{A}_{\tau,\varepsilon}^{+} \bigr)^{-1}
		\mathcal{A}_{\tau,\varepsilon}^{-} w \bigr\|_{H^{\gamma}}
		+ \frac{1}{k} \Bigl\| \bigl( \mathcal{A}_{\tau,\varepsilon}^{+} \bigr)^{-1}
		\partial_{x} P_{N} f\Bigl( \frac{w}{2} \Bigr)
		\Bigr\|_{H^{\gamma}} \\
		&\le \| w \|_{H^{\gamma}} + \tau^{\sigma} C_{2} \| w \|_{H^{\gamma}}^{k}
		\le M\bigl( 1 + \tau^{\sigma} C_{2} M^{k-1} \bigr),
	\end{align*}
	where $C_{2}>0$ is the constant arising from~\eqref{eq:Ainv-sigma} combined
	with the algebra estimate for $f(w/2)$.
	
	Choose $\tau_{M}>0$ such that
	$\tau_{M}^{\sigma}\bigl( C_{2} M^{k-1} + 2\,C_{1,M} \bigr) \le 1$.
	Then for any $0<\tau\le\tau_{M}$ and $\phi\in B_{2M}$,
	\begin{align*}
		\| \psi_{w}(\phi) \|_{H^{\gamma}}
		&\le \| \psi_{w}(0) \|_{H^{\gamma}}
		+ \| \psi_{w}(\phi) - \psi_{w}(0) \|_{H^{\gamma}} \\
		&\le M\bigl( 1 + \tau^{\sigma} C_{2} M^{k-1} \bigr)
		+ 2\,C_{1,M}\,\tau^{\sigma} M
		= M \Bigl[ 1 + \tau^{\sigma}\bigl( C_{2} M^{k-1} + 2\,C_{1,M} \bigr) \Bigr]
		\le 2M,
	\end{align*}
	so that $\psi_{w}$ maps $B_{2M}$ into itself.  Moreover,
	\begin{align}
		\| \psi_{w}(\phi_{1}) - \psi_{w}(\phi_{2}) \|_{H^{\gamma}}
		\le C_{1,M}\,\tau^{\sigma} \|\phi_{1}-\phi_{2}\|_{H^{\gamma}}
		\le \frac{1}{2} \|\phi_{1}-\phi_{2}\|_{H^{\gamma}}, \label{eq:contraction-half}
	\end{align}
	i.e., $\psi_{w}$ is a strict contraction on $B_{2M}$.  By the Banach
	fixed-point theorem, there exists a unique $\phi^{*}\in B_{2M}$ such that
	$\phi^{*} = \psi_{w}(\phi^{*})$, which completes the proof of~\textit{(ii)}.
\end{proof}

The contraction property established above directly implies the convergence of
the fixed-point iteration~\eqref{eq: fixed-point iteration} with a linear rate.

\begin{corollary}[Convergence of the fixed-point iteration]\label{crlr:iteration-convergence}
	Under the conditions of Lemma~\ref{lem:contraction-map}, let $M>0$ satisfy
	$\|u_{N}^{n}\|_{H^{\gamma}}\le M$ and let $\tau_{M}>0$ be as given therein.
	Then for all $0<\tau\le\tau_{M}$, the sequence
	$\{ u_{N,(i)}^{n+1} \}_{i=0}^{\infty}$ generated by~\eqref{eq: fixed-point iteration}
	converges to the unique solution $u_{N}^{n+1}$ in $B_{2M}$, and
	\begin{align}
		\bigl\| u_{N,(i)}^{n+1} - u_{N}^{n+1} \bigr\|_{H^{\gamma}}
		\le \frac{1}{2^{\,i-1}} \,
		\bigl\| u_{N,(1)}^{n+1} - u_{N,(0)}^{n+1} \bigr\|_{H^{\gamma}},
		\qquad i\in \mathbb{Z}_{+}. \label{eq:iter-rate}
	\end{align}
\end{corollary}
\begin{proof}
	By Lemma~\ref{lem:contraction-map}\textit{(ii)}, $\psi_{u_{N}^{n}}$ maps $B_{2M}$ to
	itself.  Starting from $u_{N,(0)}^{n+1}=u_{N}^{n}\in B_{M}\subset B_{2M}$,
	induction shows that $u_{N,(i)}^{n+1}\in B_{2M}$ for all $i\ge 0$.
	Using~\eqref{eq:contraction-half},
	\begin{align*}
		\bigl\| u_{N,(i+1)}^{n+1} - u_{N,(i)}^{n+1} \bigr\|_{H^{\gamma}}
		= \bigl\| \psi_{u_{N}^{n}}(u_{N,(i)}^{n+1})
		- \psi_{u_{N}^{n}}(u_{N,(i-1)}^{n+1}) \bigr\|_{H^{\gamma}}
		\le \frac{1}{2} \,
		\bigl\| u_{N,(i)}^{n+1} - u_{N,(i-1)}^{n+1} \bigr\|_{H^{\gamma}},
	\end{align*}
	and by iterating this inequality,
	\begin{align*}
		\bigl\| u_{N,(i+1)}^{n+1} - u_{N,(i)}^{n+1} \bigr\|_{H^{\gamma}}
		\le \frac{1}{2^{\,i}} \,
		\bigl\| u_{N,(1)}^{n+1} - u_{N,(0)}^{n+1} \bigr\|_{H^{\gamma}}.
	\end{align*}
	For any $j\ge 1$, the triangle inequality yields
	\begin{align}
		\bigl\| u_{N,(i+j)}^{n+1} - u_{N,(i)}^{n+1} \bigr\|_{H^{\gamma}}
		&\le \sum_{l=i}^{i+j-1}
		\bigl\| u_{N,(l+1)}^{n+1} - u_{N,(l)}^{n+1} \bigr\|_{H^{\gamma}} \notag \\
		&\le \sum_{l=i}^{i+j-1} \frac{1}{2^{\,l}} \,
		\bigl\| u_{N,(1)}^{n+1} - u_{N,(0)}^{n+1} \bigr\|_{H^{\gamma}}
		\le \frac{1}{2^{\,i-1}} \,
		\bigl\| u_{N,(1)}^{n+1} - u_{N,(0)}^{n+1} \bigr\|_{H^{\gamma}}
		\xrightarrow{i\to\infty} 0. \label{eq:cauchy}
	\end{align}
Thus, $\{u_{N,(i)}^{n+1}\}_{i=0}^{\infty}$ is a Cauchy sequence in the finite-dimensional
	space $Y_{N}$; denote its limit by $u_{N,(\infty)}^{n+1}$.  Then,
	\begin{align*}
		\bigl\| u_{N,(\infty)}^{n+1} - \psi_{u_{N}^{n}}(u_{N,(\infty)}^{n+1}) \bigr\|_{H^{\gamma}}
		&\le \bigl\| u_{N,(\infty)}^{n+1} - u_{N,(i+1)}^{n+1} \bigr\|_{H^{\gamma}}
		+ \bigl\| u_{N,(i+1)}^{n+1}
		- \psi_{u_{N}^{n}}(u_{N,(\infty)}^{n+1}) \bigr\|_{H^{\gamma}} \\
		&= \bigl\| u_{N,(\infty)}^{n+1} - u_{N,(i+1)}^{n+1} \bigr\|_{H^{\gamma}}
		+ \bigl\| \psi_{u_{N}^{n}}(u_{N,(i)}^{n+1})
		- \psi_{u_{N}^{n}}(u_{N,(\infty)}^{n+1}) \bigr\|_{H^{\gamma}} \\
		&\le \bigl\| u_{N,(\infty)}^{n+1} - u_{N,(i+1)}^{n+1} \bigr\|_{H^{\gamma}}
		+ \frac{1}{2} \,
		\bigl\| u_{N,(i)}^{n+1} - u_{N,(\infty)}^{n+1} \bigr\|_{H^{\gamma}}
		\xrightarrow{i\to\infty} 0,
	\end{align*}
	which shows that $u_{N,(\infty)}^{n+1} = \psi_{u_{N}^{n}}(u_{N,(\infty)}^{n+1})$.
By the uniqueness of the fixed point, $u_{N,(\infty)}^{n+1}=u_{N}^{n+1}$.
	Letting $j\to\infty$ in~\eqref{eq:cauchy} gives~\eqref{eq:iter-rate}.
\end{proof}

\begin{lemma}[Discrete conservation laws]\label{lem: conservative law}
	The fully discrete CN scheme~\eqref{scm: CN-FP} preserves discrete mass and energy, i.e.,
	\begin{align*}
		M[u_{N}^{n}] \equiv M[u_{N}^{0}], \qquad
		E[u_{N}^{n}] \equiv E[u_{N}^{0}], \qquad n = 0, 1, \dots
	\end{align*}
\end{lemma}
\begin{proof}
	Integrating~\eqref{scm: CN-FP} over $\mathbb{T}$ immediately gives
	$M[u_{N}^{n+1}] = M[u_{N}^{n}]$ for all $n\in\mathbb{N}$.
	For the energy, take the $L^{2}$ inner product of~\eqref{scm: CN-FP} with
	$u_{N}^{n+\frac{1}{2}}$:
	\begin{align*}
		\frac{1}{2\tau}\Bigl( \| u_{N}^{n+1} \|_{L^{2}}^{2} - \| u_{N}^{n} \|_{L^{2}}^{2} \Bigr)
		+ \varepsilon^{2} \bigl\langle \partial_{x}^{3} u_{N}^{n+\frac{1}{2}},
		u_{N}^{n+\frac{1}{2}} \bigr\rangle
		+ \frac{1}{k} \bigl\langle \partial_{x} P_{N} f\bigl( u_{N}^{n+\frac{1}{2}} \bigr),
		u_{N}^{n+\frac{1}{2}} \bigr\rangle = 0.
	\end{align*}
	Integration-by-parts yields
	$\langle \partial_{x}^{3} u_{N}^{n+\frac{1}{2}}, u_{N}^{n+\frac{1}{2}} \rangle = 0$.
	Moreover, since $u_{N}^{n+\frac{1}{2}}\in Y_{N}$,
	\begin{align*}
		\bigl\langle \partial_{x} P_{N} f\bigl( u_{N}^{n+\frac{1}{2}} \bigr),
		u_{N}^{n+\frac{1}{2}} \bigr\rangle
		= \bigl\langle \partial_{x} f\bigl( u_{N}^{n+\frac{1}{2}} \bigr),
		u_{N}^{n+\frac{1}{2}} \bigr\rangle = 0,
	\end{align*}
	again by integration-by-parts.  Hence,
	$\| u_{N}^{n+1} \|_{L^{2}}^{2} = \| u_{N}^{n} \|_{L^{2}}^{2}$,
	completing the proof.
\end{proof}

\subsection{Convergence result}\label{sec: 3.3}

In the dispersionless regime $\varepsilon \ll 1$, it is known that there exists a critical time $t_{c}$, independent of $\varepsilon$ but depending on the initial datum $u_0$, such that the solution $u(x,t)$ of~\eqref{eq: gKdV} remains smooth for all $t < t_{c}$ \cite{lax1983small, masoero2013semiclassical}. Beyond $t_{c}$, dispersive shock waves develop and the characteristic wavelength shrinks to $\mathcal{O}(\varepsilon)$; this highly challenging regime is deferred to future work.  Throughout this section, we consider the terminal time $T < t_{c}$, restricting the computation time to $0 \leq t \le T$, and impose the following regularity assumption (no finite-time blowup) on the exact solution of the gKdV equation \eqref{eq: gKdV}  as follows.
\begin{assumption}\label{assump: regularity for CN}
	Suppose that $u \in L^{\infty}\bigl(0, T; H^{\gamma + m}(\mathbb{T})\bigr)$,
	$\partial_{t}^{2} u \in L^{\infty}\bigl(0, T; H^{\gamma + 3}(\mathbb{T})\bigr)$, and
	$\partial_{t}^{3} u \in L^{\infty}\bigl(0, T; H^{\gamma}(\mathbb{T})\bigr)$
	with $\gamma > \frac{3}{2},m \ge 1$ and $0<T<t_c$. Moreover, there exists a constant
	$M >1$ independent of $\varepsilon$ such that
	\begin{align*}
		\max_{0\le t \le T} \bigl\{ \lVert u \rVert_{H^{\gamma + m}},\;
		\lVert \partial_{t}^{2} u \rVert_{H^{\gamma + 3}},\; \lVert \partial_{t}^{3} u \rVert_{H^{\gamma}} \bigr\} \le M-1.
	\end{align*}
\end{assumption}

\begin{theorem}[Convergence of  CN scheme]\label{thm:CN-convergence}
	Assume that  Assumption~\ref{assump: regularity for CN} holds and $\tau \lesssim \min\{ \varepsilon^{1+\sigma_{0}},\, N^{-(1+\sigma_{0})} \}$ for any fixed $\sigma_{0}>0$ that can be arbitrarily small. Let $u_{N}^{n}$ be the numerical solution computed via the CN scheme~\eqref{scm: CN-FP}.  Then, there exist $N_{1}>0$ and $\tau_{1}>0$ such that if $N\ge N_{1}$ and $0<\tau\le\tau_{1}$,
	\begin{align}
		\| u(\cdot, t_{n}) - u_{N}^{n} \|_{H^{\gamma}}
		\le C\bigl( \tau^{2} + N^{-m+1} \bigr), \qquad 0 \le n \le T/\tau, \label{eq:CN-error}
	\end{align}
	where $N_{1}, \tau_{1},C>0$ depend on $M$ and $T$, but is independent of $\varepsilon$, $N$, and $\tau$.
\end{theorem}

\begin{proof}
	We proceed in three stages. 
    
    \textbf{Step 1.}(Projected equation and local truncation error). Applying $P_{N}$ to both sides of~\eqref{eq: gKdV} and inserting the projection inside the nonlinear term,
	\begin{align*}
		\partial_{t} P_{N}u
		+ \varepsilon^{2}\partial_{x}^{3} P_{N}u
		+ \frac{1}{k} \partial_{x} P_{N} f(P_{N}u)
		= \frac{1}{k} \partial_{x} P_{N} \bigl[ f(P_{N}u) - f(u) \bigr].
	\end{align*}
	Set $U_{N}^{n} := P_{N}u(\cdot, t_{n})$.  Centering the discretization at $t_{n+\frac{1}{2}}$ and rearranging yields
	\begin{equation}\label{eq:UN}
		\frac{1}{\tau}\bigl( U_{N}^{n+1} - U_{N}^{n} \bigr)
		+ \varepsilon^{2}\partial_{x}^{3} U_{N}^{n+\frac{1}{2}}
		+ \frac{1}{k} \partial_{x} P_{N} f\bigl( U_{N}^{n+\frac{1}{2}} \bigr)
		= R_{1}^{n} + R_{2,1}^{n} + \varepsilon^{2}\partial_{x}^{3} R_{2,2}^{n}
		+ R_{2,3}^{n},
	\end{equation}
	with the projection error of the nonlinearity:
	\begin{align*}
		R_{1}^{n} &:= \frac{1}{k} \partial_{x} P_{N}
		\Bigl[ f\bigl( P_{N}u(t_{n+\frac{1}{2}}) \bigr)
		- f\bigl( u(t_{n+\frac{1}{2}}) \bigr) \Bigr] \\
		&= P_{N} \Bigl[ (P_{N}-I)u(t_{n+\frac{1}{2}}) \, Q_{1}\bigl( P_{N}u(t_{n+\frac{1}{2}}),\, u(t_{n+\frac{1}{2}}) \bigr) + \partial_{x}\bigl( (P_{N}-I)u(t_{n+\frac{1}{2}}) \bigr) \, Q_{2}\bigl( u(t_{n+\frac{1}{2}}) \bigr) \Bigr],
	\end{align*}
	the temporal quadrature errors:
	\begin{align*}
		R_{2,1}^{n} &:= \frac{1}{\tau}\bigl( U_{N}^{n+1} - U_{N}^{n} \bigr) - \partial_{t} P_{N}u(t_{n+\frac{1}{2}}) = \frac{\tau^{2}}{16} \int_{0}^{1} (1-\rho)^{2} \partial_{t}^{3} P_{N} \Bigl( u\bigl( t_{n+\frac{1}{2}} + \tfrac{\rho\tau}{2} \bigr) + u\bigl( t_{n+\frac{1}{2}} - \tfrac{\rho\tau}{2} \bigr) \Bigr) \,\mathrm{d}\rho, \\[4pt]
		R_{2,2}^{n} &:= U_{N}^{n+\frac{1}{2}} - P_{N}u(t_{n+\frac{1}{2}}) = \frac{\tau^{2}}{8} \int_{0}^{1} (1-\rho) \partial_{t}^{2} P_{N} \Bigl( u\bigl( t_{n+\frac{1}{2}} + \tfrac{\rho\tau}{2} \bigr) + u\bigl( t_{n+\frac{1}{2}} - \tfrac{\rho\tau}{2} \bigr) \Bigr) \,\mathrm{d}\rho,
	\end{align*}
	and the coupling error from replacing $u(t_{n+\frac{1}{2}})$ by $U_{N}^{n+\frac{1}{2}}$ in the nonlinear term:
	\begin{align*}
		R_{2,3}^{n} &:= \frac{1}{k} \partial_{x} P_{N}
		\Bigl[ f\bigl( U_{N}^{n+\frac{1}{2}} \bigr)
		- f\bigl( P_{N}u(t_{n+\frac{1}{2}}) \bigr) \Bigr] \\
		&= P_{N} \Bigl[ R_{2,2}^{n} \, Q_{1}\bigl( U_{N}^{n+\frac{1}{2}},\, P_{N}u(t_{n+\frac{1}{2}}) \bigr) + \partial_{x} R_{2,2}^{n} \, Q_{2}\bigl( P_{N}u(t_{n+\frac{1}{2}}) \bigr) \Bigr].
	\end{align*}
	Using Lemma~\ref{lem:operator-est} for $R_{1}^{n}$ and the integral representations for $R_{2,1}^{n}$, $R_{2,2}^{n}$, and $R_{2,3}^{n}$, we obtain the bounds
	\begin{subequations}\begin{align}
			\| R_{1}^{n} \|_{H^{\gamma}}
			\lesssim&\, N^{-m+1} \,
			\| u \|_{L^{\infty}(t_{n},t_{n+1};H^{\gamma+m})}
			\| u \|_{L^{\infty}(t_{n},t_{n+1};H^{\gamma})}^{k-1}, \label{eq:R3} \\
			\| R_{2,1}^{n} \|_{H^{\gamma}}
			\lesssim&\, \tau^{2} \,
			\| \partial_{t}^{3} u \|_{L^{\infty}(t_{n},t_{n+1};H^{\gamma})}, \label{eq:R41} \\
			\| \varepsilon^{2}\partial_{x}^{3} R_{2,2}^{n} \|_{H^{\gamma}}
			\lesssim&\, \varepsilon^{2} \tau^{2} \,
			\| \partial_{t}^{2} u \|_{L^{\infty}(t_{n},t_{n+1};H^{\gamma+3})}, \label{eq:R42} \\
			\| R_{2,3}^{n} \|_{H^{\gamma}}
			\lesssim&\, \tau^{2} \,
			\| \partial_{t}^{2} u \|_{L^{\infty}(t_{n},t_{n+1};H^{\gamma+1})}\nonumber \\
			&\, \cdot \Bigl( \| u \|_{L^{\infty}(t_{n},t_{n+1}; H^{\gamma+1})}
			\| u \|_{L^{\infty}(t_{n},t_{n+1}; H^{\gamma})}^{k-2}
			+ \| u \|_{L^{\infty}(t_{n},t_{n+1}; H^{\gamma})}^{k-1} \Bigr). \label{eq:R43}
\end{align}\end{subequations}
	Under the regularity Assumption~\ref{assump: regularity for CN}, all
	$H^{\gamma}$-norms on the right-hand sides are bounded uniformly in $N$,
	$\tau$, and $\varepsilon$ (with $\varepsilon\le 1$).  Defining
	$R_{2}^{n} := R_{2,1}^{n} + \varepsilon^{2}\partial_{x}^{3}R_{2,2}^{n}
	+ R_{2,3}^{n}$, we therefore have
	\begin{align}
		\| R_{1}^{n} \|_{H^{\gamma}} \lesssim N^{-m+1}, \qquad
		\| R_{2}^{n} \|_{H^{\gamma}} \lesssim \tau^{2}. \label{eq:R-bound}
	\end{align}
	
	\textbf{Step 2.}(Error equation, induction and energy estimate).
	Let $e_{u}^{n} := U_{N}^{n} - u_{N}^{n}$.  Using~\eqref{eq:f-f-3}, the difference of the nonlinear terms expands as
	\begin{align*}
		\frac{1}{k} \partial_{x} \Bigl[ f\bigl( U_{N}^{n+\frac{1}{2}} \bigr)
		- f\bigl( u_{N}^{n+\frac{1}{2}} \bigr) \Bigr]
		= e_{u}^{n+\frac{1}{2}} \,
		Q_{1}\bigl( U_{N}^{n+\frac{1}{2}},\,
		u_{N}^{n+\frac{1}{2}} \bigr)
		+ \partial_{x} e_{u}^{n+\frac{1}{2}} \,
		Q_{2}\bigl( u_{N}^{n+\frac{1}{2}} \bigr)
		=: G_{1}^{n}.
	\end{align*}
	Subtracting the CN scheme~\eqref{scm: CN-FP}
	from~\eqref{eq:UN} yields the error equation
	\begin{equation}\label{eq:error-eq-CN}
		\frac{1}{\tau}\bigl( e_{u}^{n+1} - e_{u}^{n} \bigr)
		+ \varepsilon^{2}\partial_{x}^{3} e_{u}^{n+\frac{1}{2}}
		+ P_{N} G_{1}^{n}
		= R_{1}^{n} + R_{2}^{n}.
	\end{equation} 
	Based on the definition of $M$, we have $\| u \|_{L^{\infty}(0,T;H^{\gamma})} + 1 \leq M$, and proceed by induction.  Suppose that
	\begin{align}
		\| u_{N}^{m} \|_{H^{\gamma}} \le M, \qquad
		0 \le m \le n \le T/\tau - 1. \label{eq:ind-hyp-CN}
	\end{align}
	By Corollary~\ref{crlr:iteration-convergence}, there exists $\tau_{M}>0$ such
	that for $\tau\le\tau_{M}$, $\|u_{N}^{n+1}\|_{H^{\gamma}}\le 2M$.
	Take the $H^{\gamma}$ inner product of~\eqref{eq:error-eq-CN} with
	$e_{u}^{n+\frac{1}{2}}$.  Since
	$\langle \partial_{x}^{3} e_{u}^{n+\frac{1}{2}},
	e_{u}^{n+\frac{1}{2}} \rangle_{H^{\gamma}} = 0$, we obtain
	\begin{align}
		\frac{1}{2\tau}\Bigl( \| e_{u}^{n+1} \|_{H^{\gamma}}^{2}
		- \| e_{u}^{n} \|_{H^{\gamma}}^{2} \Bigr)
		= -\bigl\langle P_{N} G_{1}^{n},\, e_{u}^{n+\frac{1}{2}} \bigr\rangle_{H^{\gamma}}
		+ \bigl\langle R_{1}^{n} + R_{2}^{n},\,
		e_{u}^{n+\frac{1}{2}} \bigr\rangle_{H^{\gamma}}. \label{eq:energy-CN}
	\end{align}
	To formulate compactly, we introduce the auxiliary functional
	\begin{align}
		\mathcal{Q}[w_{1}, w_{2}]
		:= \| \partial_{x} w_{1} \|_{H^{\gamma}}
		\sum_{j=0}^{k-2} \| w_{1} \|_{H^{\gamma}}^{j} \| w_{2} \|_{H^{\gamma}}^{k-2-j}
		+ \| w_{2} \|_{H^{\gamma}}^{k-1}, \qquad w_{1}, w_{2} \in Y_{N}. \label{eq:Q-func}
	\end{align}
	For the first inner product on the right side of \eqref{eq:energy-CN}, applying Lemma~\ref{lem:Kato-Ponce} and the
	definition~\eqref{eq:Q-func} of $\mathcal{Q}$,
	\begin{align}
		\left| \Big\langle P_{N}G_{1}^{n}, e_{u}^{n+\frac{1}{2}} \Big\rangle_{H^{\gamma}} \right| =&\, \left| \Big\langle e_{u}^{n+\frac{1}{2}}, e_{u}^{n+\frac{1}{2}} Q_{1}\left( U_{N}^{n+\frac{1}{2}}, u_{N}^{n+\frac{1}{2}} \right) + \partial_{x} e_{u}^{n+\frac{1}{2}} Q_{2}\left( u_{N}^{n+\frac{1}{2}} \right) \Big\rangle_{H^{\gamma}} \right| \nonumber\\
		\leq&\, \left| \left\langle e_{u}^{n+\frac{1}{2}}, e_{u}^{n+\frac{1}{2}} Q_{1}\left( U_{N}^{n+\frac{1}{2}}, u_{N}^{n+\frac{1}{2}} \right) \right\rangle_{H^{\gamma}} \right| + \left| \left\langle e_{u}^{n+\frac{1}{2}}, \partial_{x} e_{u}^{n+\frac{1}{2}} Q_{2}\left( u_{N}^{n+\frac{1}{2}} \right) \right\rangle_{H^{\gamma}} \right| \nonumber\\
		\lesssim&\, \left\| e_{u}^{n+\frac{1}{2}} \right\|_{H^{\gamma}}^{2}\left( \left\| Q_{1}\left( U_{N}^{n+\frac{1}{2}}, u_{N}^{n+\frac{1}{2}}\right) \right\|_{H^{\gamma}}  + \left\| Q_{2}\left( u_{N}^{n+\frac{1}{2}} \right) \right\|_{H^{\gamma}} \right)  \nonumber\\
		\lesssim&\, \left\| e_{u}^{n+\frac{1}{2}} \right\|_{H^{\gamma}}^{2} \mathcal{Q}\left[ U_{N}^{n+\frac{1}{2}}, u_{N}^{n+\frac{1}{2}} \right]. \label{eq:G3-est}
	\end{align}
	For the second inner product, the Cauchy--Schwarz and Young inequalities give
	\begin{align}
		\bigl| \bigl\langle R_{1}^{n} + R_{2}^{n},\,
		e_{u}^{n+\frac{1}{2}} \bigr\rangle_{H^{\gamma}} \bigr|
		&\le \bigl( \| R_{1}^{n} \|_{H^{\gamma}} + \| R_{2}^{n} \|_{H^{\gamma}} \bigr)
		\, \bigl\| e_{u}^{n+\frac{1}{2}} \bigr\|_{H^{\gamma}} \lesssim \| R_{1}^{n} \|_{H^{\gamma}}^{2}
		+ \| R_{2}^{n} \|_{H^{\gamma}}^{2}
		+ \bigl\| e_{u}^{n+\frac{1}{2}} \bigr\|_{H^{\gamma}}^{2}. \label{eq:residual-est}
	\end{align}
	Inserting~\eqref{eq:G3-est} and~\eqref{eq:residual-est}
	into~\eqref{eq:energy-CN} and using~\eqref{eq:R-bound}, we deduce the
	existence of constants $C_{3}, C_{4} > 0$ such that
	\begin{align}
		\| e_{u}^{n+1} \|_{H^{\gamma}}^{2} - \| e_{u}^{n} \|_{H^{\gamma}}^{2}
		&\le C_{3}\,\tau \,
		\bigl( \| e_{u}^{n+1} \|_{H^{\gamma}}^{2}
		+ \| e_{u}^{n} \|_{H^{\gamma}}^{2} \bigr)
		+ C_{4}\,\tau \left( \tau^{4} + N^{-2m+2} \right). \label{eq:error-rec-CN}
	\end{align}
	
	\textbf{Step 3.}(Gr\"{o}nwall argument and closure of induction).
	Since $u_{N}^{0} = P_{N} I_{2N} u_{0}$,
	The initial error satisfies,
	\begin{align}
		\| e_{u}^{0} \|_{H^{\gamma}}
		= \bigl\| P_{N}u(t_{0}) - P_{N}I_{2N}u(t_{0}) \bigr\|_{H^{\gamma}}
		\le \bigl\| (I_{2N}-I)u(t_{0}) \bigr\|_{H^{\gamma}}
		\lesssim (2N)^{-m} \| u(t_{0}) \|_{H^{\gamma+m}}. \label{eq:init-CN}
	\end{align}
	Applying the discrete Gr\"{o}nwall inequality to~\eqref{eq:error-rec-CN}
	and using~\eqref{eq:init-CN},
	\begin{align*}
		\| e_{u}^{n+1} \|_{H^{\gamma}}
		&\le \Bigl( \| e_{u}^{0} \|_{H^{\gamma}}^{2}
		+ \tau \sum_{m=1}^{n+1}
		C_{4}\bigl( \tau^{4} + N^{-2m+2} \bigr) \Bigr)^{\!1/2}
		\exp\bigl( 2\,C_{3}\,T \bigr) \le \widetilde{C}_{3} \bigl( \tau^{2} + N^{-m+1} \bigr),
	\end{align*}
	where $\widetilde{C}_{3}>0$ depends on $M$, $T$, and the constants
	$C_{3}, C_{4}$, but is independent of $\varepsilon$, $N$, and $\tau$.
	
	Choose $\tau_{1}>0$ and $N_{1}>0$ such that
	$\widetilde{C}_{3}(\tau_{1}^{2} + N_{1}^{-m+1}) \le 1$.  Then for all
	$0<\tau\le\tau_{1}$ and $N\ge N_{1}$,
	\begin{align*}
		\| u_{N}^{n+1} \|_{H^{\gamma}}
		\le \| U_{N}^{n+1} \|_{H^{\gamma}} + \| e_{u}^{n+1} \|_{H^{\gamma}}
		\le (M-1) + 1 = M,
	\end{align*}
	which restores the induction hypothesis~\eqref{eq:ind-hyp-CN} at step $n+1$. By induction,~\eqref{eq:CN-error} holds for all $0\le n\le T/\tau$.
\end{proof}

\section{Analysis of the Lawson-RK scheme}\label{sec:4}

In this section, we analyze the convergence of the fully discrete Lawson-RK scheme~\eqref{scm: RK2-FP}. We first impose the following regularity assumption on the exact solution.
\begin{assumption} \label{assump: regularity for RK2} For $\gamma > \frac{3}{2},m \ge 7$ and $0<T<t_c$, 
	suppose that $u \in L^{\infty}\bigl(0, T; H^{\gamma + m}(\mathbb{T})\bigr)$,
	$\partial_{t} u \in L^{\infty}\bigl(0, T; H^{\gamma + 4}(\mathbb{T})\bigr)$, and
	$\partial_{t}^{2} u \in L^{\infty}\bigl(0, T; H^{\gamma+1}(\mathbb{T})\bigr)$. Moreover, there exists a constant $M >1$ which is independent of $\varepsilon$, such that
	\begin{align*}
		\max_{0\le t \le T} \bigl\{ \lVert u \rVert_{H^{\gamma + m}},\;
		\lVert \partial_{t} u \rVert_{H^{\gamma + 4}},\; \lVert \partial_{t}^{2} u \rVert_{H^{\gamma+1}} \bigr\} \le M-1.
	\end{align*}
\end{assumption}

\begin{theorem}[Convergence of the Lawson-RK scheme]\label{thm:RK2-convergence}
	Assume that the solution of~\eqref{eq: gKdV} satisfies Assumption~\ref{assump: regularity for RK2} and  $\tau \lesssim N^{-2}$.  Let $v_{N}^{n}$ be the numerical solution
	computed via~\eqref{scm: RK2-FP} and $u_{N}^{n} = \fe^{-t_n\varepsilon^{2}\partial_{x}^{3}}v_{N}^{n}$.  Then,  there exist
	$N_{2}>0$ and $\tau_{2}>0$ such that for $N \ge N_{2}$ and $0<\tau\le\tau_{2}$,
	\begin{align}
		\| u(\cdot, t_{n}) - u_{N}^{n} \|_{H^{\gamma}}
		\le C\bigl( \tau^{2} + N^{-m+1} \bigr), \qquad 0 \le n \le T/\tau, \label{eq:RK2-error}
	\end{align}
	where $N_{2}, \tau_{2},C>0$ depend on $M$ and  $T$, but are independent of $\varepsilon$, $N$, and $\tau$.
\end{theorem}

Since the twisted-variable transformation $v = \fe^{t\varepsilon^{2}\partial_{x}^{3}}u$ is unitary in $H^{\gamma}$, we have $\|u(t_n)-u_N^n\|_{H^{\gamma}} = \|v(t_n)-v_N^n\|_{H^{\gamma}}$. Hence,  it suffices to analyze the error $e_{v}^{n} := V_{N}^{n} - v_{N}^{n}$, where $V_{N}^{n} := P_{N}v(t_{n})$ is the projection of the exact twisted solution.

For any $w \in Y_{N}$, introduce the two-stage flow
\begin{subequations}\begin{align}
	\Phi_{\tau}^{I}(w)   &:= w - \tau P_{N} F(t_{n}, w), \label{eq:PhiI} \\
	\Phi_{\tau}^{II}(w)  &:= w - \frac{\tau}{2} P_{N}
	\bigl[ F(t_{n}, w) + F(t_{n+1}, \Phi_{\tau}^{I}(w)) \bigr], \label{eq:PhiII}
\end{align}\end{subequations}
and so $v_{N}^{n+1} = \Phi_{\tau}^{II}(v_{N}^{n})$.
The error can then be decomposed as
\begin{align}
	e_{v}^{n+1}
	= \underbrace{V_{N}^{n+1} - \Phi_{\tau}^{II}(V_{N}^{n})}_{\text{local truncation error}}
	\;+\; \underbrace{\Phi_{\tau}^{II}(V_{N}^{n}) - \Phi_{\tau}^{II}(v_{N}^{n})}_{\text{stability}}. \label{eq:error-decomp}
\end{align}
We estimate the two contributions separately.

\begin{lemma}[Local truncation error]\label{lem:local-error-RK}
	Under the regularity Assumption~\ref{assump: regularity for RK2},
	\begin{align}
		\bigl\| V_{N}^{n+1} - \Phi_{\tau}^{II}(V_{N}^{n}) \bigr\|_{H^{\gamma}}
		\le C_{5}\tau \left[ \tau^{2} + \left( 1 + \tau + \tau N \right)N^{-m+1}\right], \label{eq:local-error-bound}
	\end{align}
	where $C_{5}>0$ depends on $M$ and $T$, but is independent of $\varepsilon$, $N$, and $\tau$.
\end{lemma}
\begin{proof}
	Applying $P_{N}$ to both sides of~\eqref{eq: gKdV for v} gives
	$\partial_{t}P_{N}v = -P_{N}F(t, v)$.  Inserting the projection inside the
	nonlinear term and using~\eqref{eq:f-f-3}, we obtain
	\begin{align}
		\partial_{t}P_{N}v
		&= -P_{N}F(t, P_{N}v)
		+ \fe^{t\varepsilon^{2}\partial_{x}^{3}} P_{N}
		\Bigl[ \fe^{-t\varepsilon^{2}\partial_{x}^{3}}(P_{N}v - v) \,
		Q_{1}\bigl( \fe^{-t\varepsilon^{2}\partial_{x}^{3}}P_{N}v,\,
		\fe^{-t\varepsilon^{2}\partial_{x}^{3}}v \bigr) \notag \\
		&\qquad + \fe^{-t\varepsilon^{2}\partial_{x}^{3}}\partial_{x}(P_{N}v - v) \,
		Q_{2}\bigl( \fe^{-t\varepsilon^{2}\partial_{x}^{3}}v \bigr) \Bigr]. \label{eq:PN-v}
	\end{align}
	Integrating~\eqref{eq:PN-v} over $[t_{n}, t_{n+1}]$ yields
	\begin{align*}
		V_{N}^{n+1}
		= V_{N}^{n} - \int_{0}^{\tau} P_{N}F\bigl(t_{n}+s, P_{N}v(t_{n}+s)\bigr)\,\mathrm{d}s
		+ R_{3}^{n},
	\end{align*}
	where the residual $R_{3}^{n}$ collects the terms involving $P_{N}v - v$.
	By Lemma~\ref{lem:operator-est} and together with the algebra property of $H^{\gamma}$ ($\gamma>3/2$), we obtain the estimate
	\begin{align}
		\| R_{3}^{n} \|_{H^{\gamma}} \lesssim \tau N^{-m+1} \, \| v \|_{L^{\infty}(t_{n},t_{n+1};H^{\gamma+m})} \| v \|_{L^{\infty}(t_{n},t_{n+1}; H^{\gamma})}^{k-1}. \label{eq:R0-est}
	\end{align}
	An analogous bound 
	\begin{align}
		\| R_{3}^{n} \|_{H^{\gamma+1}} \lesssim (\tau N) N^{-m+1} \, \| v \|_{L^{\infty}(t_{n},t_{n+1};H^{\gamma+m})} \| v \|_{L^{\infty}(t_{n},t_{n+1}; H^{\gamma})}^{k-1}. \label{eq:R0-est-plus1}
	\end{align}
	can be obtained and will be used later for the higher-order estimates.
	
	Set $\mathcal{F}(t) := F(t, P_{N} v(t))$.  Then the exact solution satisfies
	\begin{subequations}\label{eq:VN}
		\begin{align}
			V_{N}^{n+1} &= V_{N}^{n} - \tau P_{N} F(t_{n}, V_{N}^{n}) + R_{4}^{n} + R_{3}^{n}, \label{eq:VN-a} \\
			V_{N}^{n+1} &= V_{N}^{n} - \frac{\tau}{2} P_{N}\bigl( F(t_{n}, V_{N}^{n}) + F(t_{n+1}, V_{N}^{n+1}) \bigr) + R_{5,1}^{n} + R_{3}^{n}, \label{eq:VN-b}
		\end{align}
	\end{subequations}
	where the quadrature residuals are given by
	\begin{align}
		R_{4}^{n} :=&\, \int_{0}^{\tau}
		P_{N}\bigl[ F(t_{n}, V_{N}^{n}) - F(t_{n}+s, P_{N}v(t_{n}+s)) \bigr]\,\mathrm{d}s = -\int_{0}^{\tau} s \int_{0}^{1}
		P_{N}\,\partial_{t}\mathcal{F}(t_{n}+\rho s)\,\mathrm{d}\rho\,\mathrm{d}s, \label{eq:R1} \\[4pt]
		R_{5,1}^{n} :=&\, \int_{0}^{\tau}
		P_{N}\Bigl[ \frac{\tau-s}{\tau}F(t_{n}, V_{N}^{n})
		+ \frac{s}{\tau}F(t_{n+1}, V_{N}^{n+1})
		- F(t_{n}+s, P_{N}v(t_{n}+s)) \Bigr]\,\mathrm{d}s \notag \\
		=&\, \int_{0}^{\tau} \frac{(\tau-s)s^{2}}{2\tau}
		\int_{0}^{1} (1-\rho) P_{N}\,\partial_{tt}\mathcal{F}
		\bigl(t_{n}+(1-\rho)s\bigr)\,\mathrm{d}\rho\,\mathrm{d}s \notag \\
		&\, + \int_{0}^{\tau} \frac{(\tau-s)^{2}s}{2\tau}
		\int_{0}^{1} (1-\rho) P_{N}\,\partial_{tt}\mathcal{F}
		\bigl(t_{n}+s+\rho(\tau-s)\bigr)\,\mathrm{d}\rho\,\mathrm{d}s, \label{eq:R21}
	\end{align}
	with the time derivatives of $\mathcal{F}$
	\begin{align*}
		\partial_{t}\mathcal{F}(t) =&\, \frac{\varepsilon^{2}}{k} \fe^{t \varepsilon^{2}\partial_{x}^{3}} \partial_{x}^{4} \left( \fe^{-t \varepsilon^{2}\partial_{x}^{3}} P_{N} v(t) \right)^{k} + \fe^{t \varepsilon^{2}\partial_{x}^{3}} \partial_{x} \Big[ \left( \fe^{-t \varepsilon^{2}\partial_{x}^{3}} P_{N} v(t) \right)^{k-1}\Big( -\varepsilon^{2}\partial_{x}^{3} \fe^{-t \varepsilon^{2}\partial_{x}^{3}} P_{N}v(t) \\
		&\, + \fe^{-t \varepsilon^{2}\partial_{x}^{3}} P_{N} \partial_{t} v(t) \Big) \Big], \\
		\partial_{tt}\mathcal{F}(t) =&\, \frac{\varepsilon^{4}}{k} \fe^{t \varepsilon^{2}\partial_{x}^{3}} \partial_{x}^{7} \left( \fe^{-t \varepsilon^{2}\partial_{x}^{3}} P_{N} v(t) \right)^{k} + \frac{\varepsilon^{2}}{k} \fe^{t \varepsilon^{2}\partial_{x}^{3}} \partial_{x}^{4} \Big[ \left( \fe^{-t \varepsilon^{2}\partial_{x}^{3}} P_{N} v(t) \right)^{k-1}\Big( -\varepsilon^{2}\partial_{x}^{3} \fe^{-t \varepsilon^{2}\partial_{x}^{3}} P_{N}v(t) \\
		&\, + \fe^{-t \varepsilon^{2}\partial_{x}^{3}} P_{N} \partial_{t} v(t) \Big) \Big] + \varepsilon^{2} \fe^{t \varepsilon^{2}\partial_{x}^{3}} \partial_{x}^{4} \Big[ \left( \fe^{-t \varepsilon^{2}\partial_{x}^{3}} P_{N} v(t) \right)^{k-1}\Big( -\varepsilon^{2}\partial_{x}^{3} \fe^{-t \varepsilon^{2}\partial_{x}^{3}} P_{N}v(t) \\
		&\, + \fe^{-t \varepsilon^{2}\partial_{x}^{3}} P_{N} \partial_{t} v(t) \Big) \Big] + (k-1) \fe^{t \varepsilon^{2}\partial_{x}^{3}}\partial_{x} \Big[ \left( \fe^{-t \varepsilon^{2}\partial_{x}^{3}} P_{N} v(t) \right)^{k-2}\Big( -\varepsilon^{2}\partial_{x}^{3} \fe^{-t \varepsilon^{2}\partial_{x}^{3}} P_{N}v(t) \\
		&\, + \fe^{-t \varepsilon^{2}\partial_{x}^{3}} P_{N} \partial_{t} v(t) \Big)^{2} \Big] + \fe^{t \varepsilon^{2}\partial_{x}^{3}} \partial_{x} \Big[ \left( \fe^{-t \varepsilon^{2}\partial_{x}^{3}} P_{N} v(t) \right)^{k-1}\Big( -\varepsilon^{4}\partial_{x}^{6} \fe^{-t \varepsilon^{2}\partial_{x}^{3}} P_{N}v(t) \\
		&\, - 2\varepsilon^{2}\partial_{x}^{3} \fe^{-t \varepsilon^{2}\partial_{x}^{3}} P_{N} \partial_{t} v(t) + \fe^{-t \varepsilon^{2}\partial_{x}^{3}} P_{N} \partial_{tt} v(t) \Big) \Big].
	\end{align*}
	From~\eqref{eq:R1}--\eqref{eq:R21}, the isometry of $\fe^{\pm t\varepsilon^{2}\partial_{x}^{3}}$ and the algebra property of $H^{\gamma}$, we obtain 
	\begin{align}
		\left\| R_{4}^{n} \right\|_{H^{\gamma}} \lesssim&\, \tau^{2} \left\| \partial_{t} \mathcal{F}(t) \right\|_{L^{\infty}\left( t_{n}, t_{n+1}; H^{\gamma} \right)} \nonumber\\
		\lesssim &\, \tau^{2} \Big( \varepsilon^{2} \left\| P_{N}v \right\|_{L^{\infty}\left( t_{n}, t_{n+1}; H^{\gamma+4} \right)}^{k} + \varepsilon^{2}\left\| P_{N}v \right\|_{L^{\infty}\left( t_{n}, t_{n+1}; H^{\gamma+1} \right)}^{k-1}\left\| P_{N}v \right\|_{L^{\infty}\left( t_{n}, t_{n+1}; H^{\gamma+4} \right)} \nonumber\\
		&\, + \left\| P_{N}v \right\|_{L^{\infty}\left( t_{n}, t_{n+1}; H^{\gamma+1} \right)}^{k-1} \left\| P_{N}\partial_{t} v \right\|_{L^{\infty}\left( t_{n}, t_{n+1}; H^{\gamma+1} \right)} \Big), \nonumber\\
		% ...
		\left\| R_{5,1}^{n} \right\|_{H^{\gamma}} \lesssim&\, \tau^{3} \left\| \partial_{tt} \mathcal{F}(t) \right\|_{L^{\infty}\left( t_{n}, t_{n+1}; H^{\gamma} \right)}   \nonumber\\
		\lesssim&\, \tau^{3}\varepsilon^{4}\Big( \left\| P_{N}v \right\|_{L^{\infty}\left( t_{n}, t_{n+1}; H^{\gamma+7} \right)}^{k} + \left\| P_{N}v \right\|_{L^{\infty}\left( t_{n}, t_{n+1}; H^{\gamma+4} \right)}^{k-1}\left\| P_{N}v \right\|_{L^{\infty}\left( t_{n}, t_{n+1}; H^{\gamma+7} \right)}  \nonumber\\
		&\, + \left\| P_{N}v \right\|_{L^{\infty}\left( t_{n}, t_{n+1}; H^{\gamma+1} \right)}^{k-2} \left\| P_{N}v \right\|_{L^{\infty}\left( t_{n}, t_{n+1}; H^{\gamma+4} \right)}^{2} \nonumber\\
		&\, + \left\| P_{N}v \right\|_{L^{\infty}\left( t_{n}, t_{n+1}; H^{\gamma+1} \right)}^{k-1} \left\| P_{N}v \right\|_{L^{\infty}\left( t_{n}, t_{n+1}; H^{\gamma+7} \right)}  \Big) \nonumber\\
		&\, + \tau^{3}\varepsilon^{2} \Big( \left\| P_{N}v \right\|_{L^{\infty}\left( t_{n}, t_{n+1}; H^{\gamma+4} \right)}^{k-1} + \left\| P_{N}v \right\|_{L^{\infty}\left( t_{n}, t_{n+1}; H^{\gamma+1} \right)}^{k-1}  \Big) \left\| P_{N}\partial_{t}v \right\|_{L^{\infty}\left( t_{n}, t_{n+1}; H^{\gamma+4} \right)} \nonumber\\
		&\, + \tau^{3} \left\| P_{N}v \right\|_{L^{\infty}\left( t_{n}, t_{n+1}; H^{\gamma+1} \right)}^{k-2} \Big( \left\| P_{N}\partial_{t}v \right\|_{L^{\infty}\left( t_{n}, t_{n+1}; H^{\gamma+1} \right)}^{2} \nonumber\\
		&\, + \left\| P_{N}v \right\|_{L^{\infty}\left( t_{n}, t_{n+1}; H^{\gamma+1} \right)}\left\| P_{N}\partial_{tt}v \right\|_{L^{\infty}\left( t_{n}, t_{n+1}; H^{\gamma+1} \right)}  \Big). \label{eq:R21-est}
	\end{align}
	Under the regularity Assumption~\ref{assump: regularity for RK2}, all
	$H^{\gamma}$-norms appearing above are bounded uniformly in $N$, $\tau$,
	and $\varepsilon$ (note that $\varepsilon\le 1$).  Consequently, $\| R_{4}^{n} \|_{H^{\gamma}} \lesssim \tau^{2}$ and $\| R_{5,1}^{n} \|_{H^{\gamma}} \lesssim \tau^{3}$.
	
	Subtracting $\Phi_{\tau}^{I}(V_{N}^{n})$ from~\eqref{eq:VN-a} and
	$\Phi_{\tau}^{II}(V_{N}^{n})$ from~\eqref{eq:VN-b} yields
	\begin{subequations}\begin{align}
		V_{N}^{n+1} - \Phi_{\tau}^{I}(V_{N}^{n}) &= R_{4}^{n} + R_{3}^{n}, \label{eq:diff-PhiI} \\
		V_{N}^{n+1} - \Phi_{\tau}^{II}(V_{N}^{n})
		&= -\frac{\tau}{2} P_{N}\bigl[ F(t_{n+1}, V_{N}^{n+1})
		- F(t_{n+1}, \Phi_{\tau}^{I}(V_{N}^{n})) \bigr]
		+ R_{5,1}^{n} + R_{3}^{n}. \label{eq:diff-PhiII}
	\end{align}\end{subequations}
	The difference of the nonlinear terms in~\eqref{eq:diff-PhiII} is expanded using~\eqref{eq:f-f-3} together with~\eqref{eq:diff-PhiI}:
	\begin{align}
		&\, F(t_{n+1}, V_{N}^{n+1}) - F(t_{n+1}, \Phi_{\tau}^{I}(V_{N}^{n})) \notag \\
		=&\, \fe^{t_{n+1}\varepsilon^{2}\partial_{x}^{3}}
		\Bigl[ \fe^{-t_{n+1}\varepsilon^{2}\partial_{x}^{3}}(R_{4}^{n}+R_{3}^{n}) \,
		Q_{1}\bigl( \fe^{-t_{n+1}\varepsilon^{2}\partial_{x}^{3}} V_{N}^{n+1},\,
		\fe^{-t_{n+1}\varepsilon^{2}\partial_{x}^{3}} \Phi_{\tau}^{I}(V_{N}^{n}) \bigr) \notag \\
		&\, + \fe^{-t_{n+1}\varepsilon^{2}\partial_{x}^{3}}
		\partial_{x}(R_{4}^{n}+R_{3}^{n}) \,
		Q_{2}\bigl( \fe^{-t_{n+1}\varepsilon^{2}\partial_{x}^{3}}
		\Phi_{\tau}^{I}(V_{N}^{n}) \bigr) \Bigr]
		=: R_{5,2}^{n}. \label{eq:R22}
	\end{align}
	A direct estimate gives
	\begin{align}
		\| R_{5,2}^{n} \|_{H^{\gamma}}
		&\lesssim \bigl( \| R_{4}^{n} \|_{H^{\gamma}} + \| R_{3}^{n} \|_{H^{\gamma}} \bigr)
		\| \partial_{x} V_{N}^{n+1} \|_{H^{\gamma}}
		\sum_{j=0}^{k-2} \| V_{N}^{n+1} \|_{H^{\gamma}}^{j}
		\| \Phi_{\tau}^{I}(V_{N}^{n}) \|_{H^{\gamma}}^{k-2-j} \notag \\
		&\quad + \bigl( \| R_{4}^{n} \|_{H^{\gamma+1}} + \| R_{3}^{n} \|_{H^{\gamma+1}} \bigr)
		\| \Phi_{\tau}^{I}(V_{N}^{n}) \|_{H^{\gamma}}^{k-1}. \label{eq:R22-est}
	\end{align}
	Substituting~\eqref{eq:R22} into~\eqref{eq:diff-PhiII}, we obtain
	\begin{align*}
		V_{N}^{n+1} - \Phi_{\tau}^{II}(V_{N}^{n})
		= -\frac{\tau}{2} P_{N} R_{5,2}^{n} + R_{5,1}^{n} + R_{3}^{n}.
	\end{align*}
	Combining~\eqref{eq:R0-est}, \eqref{eq:R0-est-plus1}, \eqref{eq:R21-est}, and~\eqref{eq:R22-est} yields the desired bound~\eqref{eq:local-error-bound}.
\end{proof}

\begin{lemma}[Stability of the Lawson-RK flow]\label{lem:stability-RK}
	Under the assumption of Theorem~\ref{thm:RK2-convergence}, and the induction hypothesis that $\|v_{N}^{n}\|_{H^{\gamma}}$ and $\|V_{N}^{n}\|_{H^{\gamma}}$ are bounded by a constant $M$ (see the proof of Theorem~\ref{thm:RK2-convergence}), there exists $C_{6}>0$ such that
	\begin{align}
		\bigl\| \Phi_{\tau}^{II}(V_{N}^{n}) - \Phi_{\tau}^{II}(v_{N}^{n}) \bigr\|_{H^{\gamma}}
		\le \bigl( 1 + C_{6}\,\tau \bigr) \| e_{v}^{n} \|_{H^{\gamma}}, \label{eq:stab-bound}
	\end{align}
	where $C_{6}$ depends on $M$, but is independent of $\varepsilon$, $N$, and $\tau$.
\end{lemma}
\begin{proof}
	We proceed in two stages, corresponding to the two stages of the Lawson-RK flow. 
    
    \textbf{Stage 1.}  From~\eqref{eq:PhiI},
	\begin{align}
		\Phi_{\tau}^{I}(V_{N}^{n}) - \Phi_{\tau}^{I}(v_{N}^{n})
		= e_{v}^{n} - \tau P_{N} \bigl[ F(t_{n}, V_{N}^{n}) - F(t_{n}, v_{N}^{n}) \bigr]
		=: e_{v}^{n} - \tau P_{N} G_{2}^{n}. \label{eq:PhiI-diff}
	\end{align}
	Using~\eqref{eq:f-f-3}, $G_{2}^{n}$ expands as
	\begin{align*}
		G_{2}^{n}
		&= \fe^{t_{n}\varepsilon^{2}\partial_{x}^{3}}
		\Bigl[ \fe^{-t_{n}\varepsilon^{2}\partial_{x}^{3}} e_{v}^{n} \,
		Q_{1}\bigl( \fe^{-t_{n}\varepsilon^{2}\partial_{x}^{3}} V_{N}^{n},\,
		\fe^{-t_{n}\varepsilon^{2}\partial_{x}^{3}} v_{N}^{n} \bigr) + \fe^{-t_{n}\varepsilon^{2}\partial_{x}^{3}} \partial_{x} e_{v}^{n} \,
		Q_{2}\bigl( \fe^{-t_{n}\varepsilon^{2}\partial_{x}^{3}} v_{N}^{n} \bigr) \Bigr].
	\end{align*}
	Taking the $H^{\gamma}$-norm squared of~\eqref{eq:PhiI-diff},
	\begin{align}
		\| \Phi_{\tau}^{I}(V_{N}^{n}) - \Phi_{\tau}^{I}(v_{N}^{n}) \|_{H^{\gamma}}^{2}
		= \| e_{v}^{n} \|_{H^{\gamma}}^{2}
		- 2\tau \langle e_{v}^{n}, P_{N} G_{2}^{n} \rangle_{H^{\gamma}}
		+ \tau^{2} \| P_{N} G_{2}^{n} \|_{H^{\gamma}}^{2}. \label{eq:PhiI-sq}
	\end{align}
	For the inner product term, applying Lemma~\ref{lem:Kato-Ponce} and the
	definition~\eqref{eq:Q-func} of $\mathcal{Q}$ yields
	\begin{align}
		\left\langle e_{v}^{n}, P_{N} G_{2}^{n} \right\rangle_{H^{\gamma}} =&\, \left\langle \fe^{-t_{n}\varepsilon^{2}\partial_{x}^{3}}e_{v}^{n}, \fe^{-t_{n}\varepsilon^{2}\partial_{x}^{3}} e_{v}^{n} \cdot Q_{1}\left( \fe^{-t_{n}\varepsilon^{2}\partial_{x}^{3}} V_{N}^{n}, \fe^{-t_{n}\varepsilon^{2}\partial_{x}^{3}} v_{N}^{n} \right) \right\rangle_{H^{\gamma}}  \nonumber\\
		&\, + \left\langle \fe^{-t_{n}\varepsilon^{2}\partial_{x}^{3}}e_{v}^{n}, \fe^{-t_{n}\varepsilon^{2}\partial_{x}^{3}} \partial_{x} e_{v}^{n} \cdot Q_{2}\left( \fe^{-t_{n}\varepsilon^{2}\partial_{x}^{3}} v_{N}^{n} \right) \right\rangle_{H^{\gamma}} \lesssim \left\| e_{v}^{n} \right\|_{H^{\gamma}}^{2} \mathcal{Q}[V_{N}^{n}, v_{N}^{n}]. \label{eq:inner-G1}
	\end{align}
	For the quadratic term, the Bernstein inequality $\|\partial_{x} w\|_{H^{\gamma}}
	\lesssim N \|w\|_{H^{\gamma}}$ on $Y_{N}$ gives
	\begin{align}
		\| P_{N} G_{2}^{n} \|_{H^{\gamma}}^{2}
		\lesssim N^{2} \| e_{v}^{n} \|_{H^{\gamma}}^{2}
		\bigl( \mathcal{Q}[V_{N}^{n}, v_{N}^{n}] \bigr)^{2}. \label{eq:quad-G1}
	\end{align}
	Substituting~\eqref{eq:inner-G1}--\eqref{eq:quad-G1} into~\eqref{eq:PhiI-sq}, we find
	\begin{align}
		\| \Phi_{\tau}^{I}(V_{N}^{n}) - \Phi_{\tau}^{I}(v_{N}^{n}) \|_{H^{\gamma}}^{2}
		- \| e_{v}^{n} \|_{H^{\gamma}}^{2}
		\lesssim \tau \| e_{v}^{n} \|_{H^{\gamma}}^{2} \,
		\mathcal{Q}[V_{N}^{n}, v_{N}^{n}]
		\bigl( 1 + \tau N^{2} \, \mathcal{Q}[V_{N}^{n}, v_{N}^{n}] \bigr). \label{eq:PhiI-final}
	\end{align}
	
	\textbf{Stage 2.}  From~\eqref{eq:PhiII},
	\begin{align}
		\Phi_{\tau}^{II}(V_{N}^{n}) - \Phi_{\tau}^{II}(v_{N}^{n})
		= e_{v}^{n} - \frac{\tau}{2} P_{N} G_{2}^{n}
		- \frac{\tau}{2} P_{N} G_{3}^{n}, \label{eq:PhiII-diff}
	\end{align}
	where $G_{3}^{n} := F(t_{n+1}, \Phi_{\tau}^{I}(V_{N}^{n})) - F(t_{n+1}, \Phi_{\tau}^{I}(v_{N}^{n}))$.
	Expanding $G_{3}^{n}$ via~\eqref{eq:f-f-3} and inserting~\eqref{eq:PhiI-diff},
	\begin{align}
		G_{3}^{n}
		=&\, \fe^{t_{n+1}\varepsilon^{2}\partial_{x}^{3}}
		\Bigl[ \fe^{-t_{n+1}\varepsilon^{2}\partial_{x}^{3}} e_{v}^{n} \,
		Q_{1}\bigl( \fe^{-t_{n+1}\varepsilon^{2}\partial_{x}^{3}} \Phi_{\tau}^{I}(V_{N}^{n}),\,
		\fe^{-t_{n+1}\varepsilon^{2}\partial_{x}^{3}} \Phi_{\tau}^{I}(v_{N}^{n}) \bigr) \notag \\
		&\, + \fe^{-t_{n+1}\varepsilon^{2}\partial_{x}^{3}} \partial_{x} e_{v}^{n} \,
		Q_{2}\bigl( \fe^{-t_{n+1}\varepsilon^{2}\partial_{x}^{3}}
		\Phi_{\tau}^{I}(v_{N}^{n}) \bigr) \Bigr] \notag \\
		&\, - \tau \fe^{t_{n+1}\varepsilon^{2}\partial_{x}^{3}}
		\Bigl[ \fe^{-t_{n+1}\varepsilon^{2}\partial_{x}^{3}} P_{N} G_{2}^{n} \,
		Q_{1}\bigl( \fe^{-t_{n+1}\varepsilon^{2}\partial_{x}^{3}}
		\Phi_{\tau}^{I}(V_{N}^{n}),\,
		\fe^{-t_{n+1}\varepsilon^{2}\partial_{x}^{3}}
		\Phi_{\tau}^{I}(v_{N}^{n}) \bigr) \notag \\
		&\, + \fe^{-t_{n+1}\varepsilon^{2}\partial_{x}^{3}}
		\partial_{x}(P_{N} G_{2}^{n}) \,
		Q_{2}\bigl( \fe^{-t_{n+1}\varepsilon^{2}\partial_{x}^{3}}
		\Phi_{\tau}^{I}(v_{N}^{n}) \bigr) \Bigr]
		=: G_{3,1}^{n} - \tau G_{3,2}^{n}. \label{eq:G2-decomp}
	\end{align}
	Consequently,
	\begin{align*}
		\Phi_{\tau}^{II}(V_{N}^{n}) - \Phi_{\tau}^{II}(v_{N}^{n})
		= e_{v}^{n} - \frac{\tau}{2} P_{N}(G_{2}^{n} + G_{3,1}^{n})
		+ \frac{\tau^{2}}{2} P_{N} G_{3,2}^{n}.
	\end{align*}
	Taking the $H^{\gamma}$-norm squared and expanding the right hand side, we obtain
	\begin{align}
		\left\| \Phi_{\tau}^{II}(V_{N}^{n}) - \Phi_{\tau}^{II}(v_{N}^{n}) \right\|_{H^{\gamma}}^{2} =&\, \left\| e_{v}^{n} \right\|_{H^{\gamma}}^{2} - \tau \left\langle e_{v}^{n}, P_{N}(G_{2}^{n} + G_{3,1}^{n}) \right\rangle_{H^{\gamma}} + \tau^{2} \bigg( \frac{1}{4}\left\| P_{N} \left( G_{2}^{n} + G_{3,1}^{n} \right) \right\|_{H^{\gamma}}^{2} \nonumber\\
		&\, + \left\langle e_{v}^{n}, P_{N} G_{3,2}^{n} \right\rangle_{H^{\gamma}} \bigg) - \frac{\tau^{3}}{2}\left\langle P_{N} \left( G_{2}^{n} + G_{3,1}^{n} \right), P_{N} G_{3,2}^{n} \right\rangle_{H^{\gamma}} + \frac{\tau^{4}}{4}\left\| P_{N} G_{3,2}^{n} \right\|_{H^{\gamma}}^{2}. \label{eq:PhiII-sq}
	\end{align}
	Each term is estimated using Lemma~\ref{lem:Kato-Ponce}, the Bernstein inequality, and the induction hypothesis.  We summarize the resulting bounds:
	\begin{align}
		\langle e_{v}^{n}, P_{N}(G_{2}^{n}+G_{3,1}^{n}) \rangle_{H^{\gamma}}
		&\lesssim \| e_{v}^{n} \|_{H^{\gamma}}^{2}
		\bigl( \mathcal{Q}[V_{N}^{n}, v_{N}^{n}]
		+ \mathcal{Q}[\Phi_{\tau}^{I}(V_{N}^{n}), \Phi_{\tau}^{I}(v_{N}^{n})] \bigr), \label{eq:inner-sum} \\[4pt]
		\tau^{2} \langle e_{v}^{n}, P_{N} G_{3,2}^{n} \rangle_{H^{\gamma}}
		&\lesssim \tau \| e_{v}^{n} \|_{H^{\gamma}}^{2}
		+ \tau (\tau N^{2})^{2} \| e_{v}^{n} \|_{H^{\gamma}}^{2}
		\bigl( \mathcal{Q}[V_{N}^{n}, v_{N}^{n}] \bigr)^{2}
		\bigl( \mathcal{Q}[\Phi_{\tau}^{I}(V_{N}^{n}), \Phi_{\tau}^{I}(v_{N}^{n})] \bigr)^{2}, \label{eq:inner-G22} \\[4pt]
		\text{remaining terms}
		&\lesssim \tau (\tau N^{2}) \| e_{v}^{n} \|_{H^{\gamma}}^{2}
		\Bigl( \bigl( \mathcal{Q}[V_{N}^{n}, v_{N}^{n}] \bigr)^{2}
		+ \bigl( \mathcal{Q}[\Phi_{\tau}^{I}(V_{N}^{n}), \Phi_{\tau}^{I}(v_{N}^{n})] \bigr)^{2} \Bigr) \notag \\
		&\quad + \tau^{2} (\tau N^{2})^{2} \| e_{v}^{n} \|_{H^{\gamma}}^{2}
		\bigl( \mathcal{Q}[V_{N}^{n}, v_{N}^{n}] \bigr)^{2}
		\bigl( \mathcal{Q}[\Phi_{\tau}^{I}(V_{N}^{n}), \Phi_{\tau}^{I}(v_{N}^{n})] \bigr)^{2}. \label{eq:remaining}
	\end{align}
	Substituting~\eqref{eq:inner-sum}--\eqref{eq:remaining} into~\eqref{eq:PhiII-sq}, we obtain
	\begin{align}
		\bigl\| \Phi_{\tau}^{II}(V_{N}^{n}) - \Phi_{\tau}^{II}(v_{N}^{n}) \bigr\|_{H^{\gamma}}^{2}
		\le \bigl( 1 + \widetilde{C}_{6}\,\tau \bigr) \| e_{v}^{n} \|_{H^{\gamma}}^{2},
	\end{align}
	where $\widetilde{C}_{6}$ encapsulates all quantities involving $\mathcal{Q}[\,\cdot\,{,}\,\cdot\,]$, which are bounded under the induction hypothesis and the step-size condition $\tau N^{2} \lesssim 1$. Taking the square root and using $\sqrt{1+z} \le 1 + \frac{1}{2}z$ for $z\ge0$ yields~\eqref{eq:stab-bound} with $C_{6} = \widetilde{C}_{6}/2$.
\end{proof}

\begin{proof}[Proof of Theorem~\ref{thm:RK2-convergence}]
	We proceed by an induction on $n$.  Suppose that
	\begin{align}
		\| v_{N}^{m} \|_{H^{\gamma}} \le M, \qquad 0 \le m \le n \le T/\tau - 1. \label{eq:ind-hyp}
	\end{align}
	By the isometry of $\fe^{\pm t\varepsilon^{2}\partial_{x}^{3}}$, $\|V_{N}^{m}\|_{H^{\gamma}} \le \|v(t_m)\|_{H^{\gamma}} = \|u(t_m)\|_{H^{\gamma}} \le M-1 < M$, so the induction hypothesis also covers the exact projection.
	
	Under~\eqref{eq:ind-hyp} and the condition $\tau N^{2} \lesssim 1$, the quantities $\mathcal{Q}[V_{N}^{n}, v_{N}^{n}]$ and $\mathcal{Q}[\Phi_{\tau}^{I}(V_{N}^{n}), \Phi_{\tau}^{I}(v_{N}^{n})]$ are bounded by constants depending only on $M$ and $k$.  Consequently, Lemma~\ref{lem:stability-RK} provides a constant $C_{6}>0$ such that
	\begin{align*}
		\bigl\| \Phi_{\tau}^{II}(V_{N}^{n}) - \Phi_{\tau}^{II}(v_{N}^{n}) \bigr\|_{H^{\gamma}}
		\le \bigl( 1 + C_{6}\,\tau \bigr) \| e_{v}^{n} \|_{H^{\gamma}}.
	\end{align*}
	Combining this with the error decomposition~\eqref{eq:error-decomp} and the local truncation estimate~\eqref{eq:local-error-bound} from Lemma~\ref{lem:local-error-RK}, we find
	\begin{align}
		\| e_{v}^{n+1} \|_{H^{\gamma}}
		&\le \bigl\| \Phi_{\tau}^{II}(V_{N}^{n}) - \Phi_{\tau}^{II}(v_{N}^{n}) \bigr\|_{H^{\gamma}}
		+ \bigl\| V_{N}^{n+1} - \Phi_{\tau}^{II}(V_{N}^{n}) \bigr\|_{H^{\gamma}} \notag \\
		&\le \bigl( 1 + C_{6}\,\tau \bigr) \| e_{v}^{n} \|_{H^{\gamma}}
		+ C_{5} \tau \left( \tau^{2} + N^{-m+1} \right). \label{eq:error-recur}
	\end{align} 
	For the initial error, using $u_{N}^{0} = P_{N} I_{2N} u_{0}$ and the isometry of the twisted-variable transformation,
	\begin{align}
		\| e_{v}^{0} \|_{H^{\gamma}}
		=&\, \bigl\| \fe^{t_{0}\varepsilon^{2}\partial_{x}^{3}}
		\bigl( P_{N}u(t_{0}) - u_{N}^{0} \bigr) \bigr\|_{H^{\gamma}}
		= \bigl\| P_{N}u(t_{0}) - P_{N}I_{2N}u(t_{0}) \bigr\|_{H^{\gamma}} \lesssim (2N)^{-m} \| u(t_{0}) \|_{H^{\gamma+m}}. \label{eq:init-error}
	\end{align}
	Applying the discrete Gr\"{o}nwall inequality to~\eqref{eq:error-recur} and
	using~\eqref{eq:init-error}, we get
	\begin{align*}
		\| e_{v}^{n+1} \|_{H^{\gamma}}
		&\le \Bigl( \| e_{v}^{0} \|_{H^{\gamma}}
		+ \tau \sum_{m=1}^{n+1} C_{5} \left( \tau^{2} + N^{-m+1} \right) \Bigr)
		\exp\bigl( C_{6} T \bigr) \le \widetilde{C}_{6} \bigl( \tau^{2} + N^{-m+1} \bigr),
	\end{align*}
	where $\widetilde{C}_{6}>0$ depends on $M$, $T$, but is independent of $\varepsilon$, $N$, and $\tau$.
	
	To close the induction, choose $\tau_{2}>0$ and $N_{2}>0$ such that
	$\widetilde{C}_{6}(\tau_{2}^{2} + N_{2}^{-m+1}) \le 1$.  Then for all
	$0<\tau\le\tau_{2}$ and $N\ge N_{2}$,
	\begin{align*}
		\| v_{N}^{n+1} \|_{H^{\gamma}}
		\le \| V_{N}^{n+1} \|_{H^{\gamma}} + \| e_{v}^{n+1} \|_{H^{\gamma}}
		\le (M-1) + 1 = M,
	\end{align*}
	which restores the induction hypothesis~\eqref{eq:ind-hyp} at step $n+1$.
	Thus, the estimate holds for all $0\le n \le T/\tau - 1$.
	Transforming back to the original variable,
	$\|u(\cdot,t_n) - u_N^n\|_{H^{\gamma}} = \|e_v^n\|_{H^{\gamma}}$ and
	\eqref{eq:RK2-error} follows. 
\end{proof}

\section{Numerical Experiments}\label{sec:5}

In this section, we present a series of numerical experiments to validate the theoretical analysis developed in the preceding sections.  Throughout, the spatial discretization is carried out using the Fourier pseudo-spectral method described in Section~\ref{sec:2}.  We fix $k = 5$, i.e., the $L^2$-critical case, for which the model is non-integrable and  challenging to solve numerically \cite{blowupDG,klein2015numerical}. The results of our concern for other values of $k$ remain essentially the same. 

\begin{figure}[t!]
	\centering
	\begin{subfigure}[b]{0.32\textwidth}
		\includegraphics[width=\linewidth]{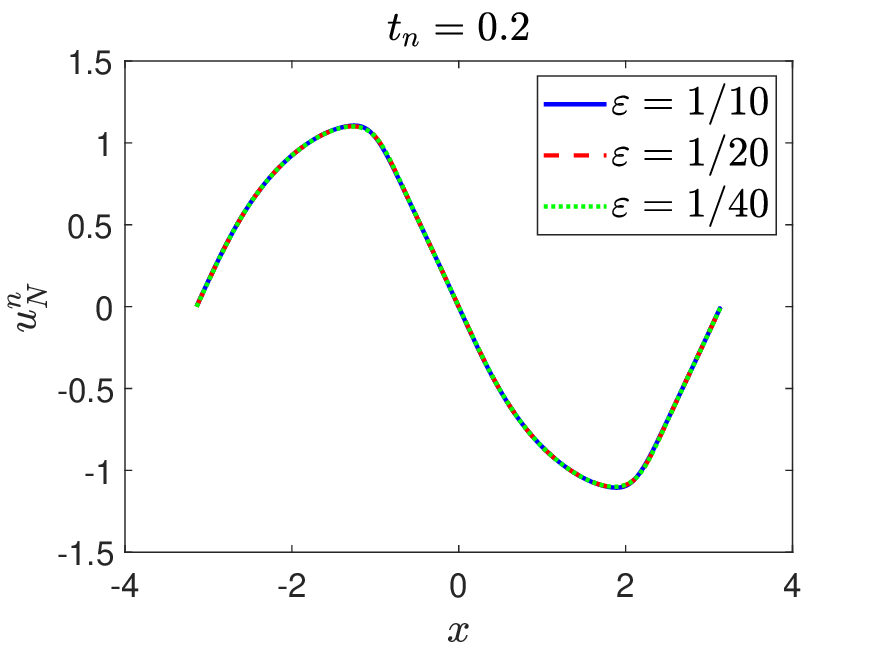}
	\end{subfigure}
	\begin{subfigure}[b]{0.32\textwidth}
		\includegraphics[width=\linewidth]{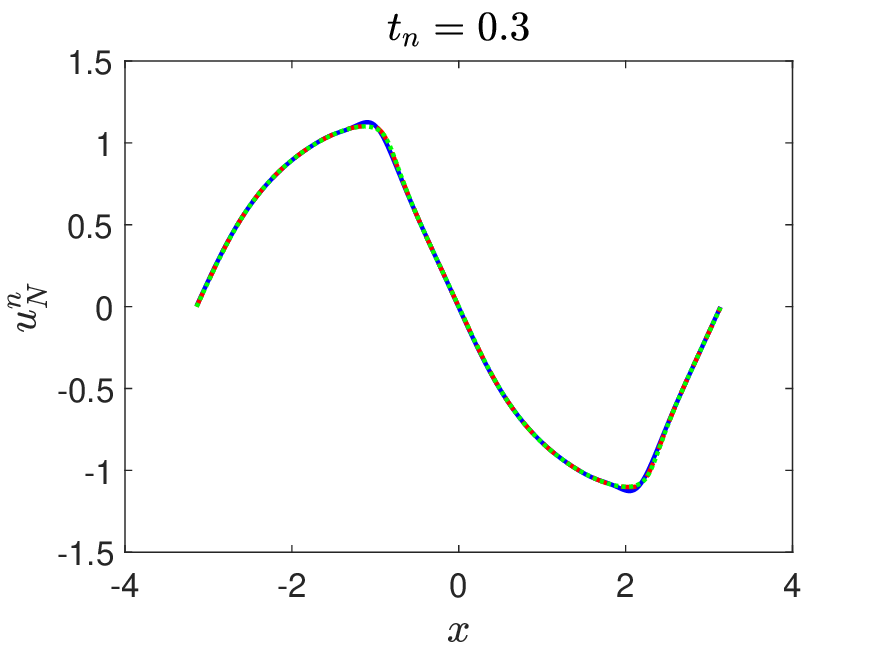}
	\end{subfigure} 
	\begin{subfigure}[b]{0.32\textwidth}
		\includegraphics[width=\linewidth]{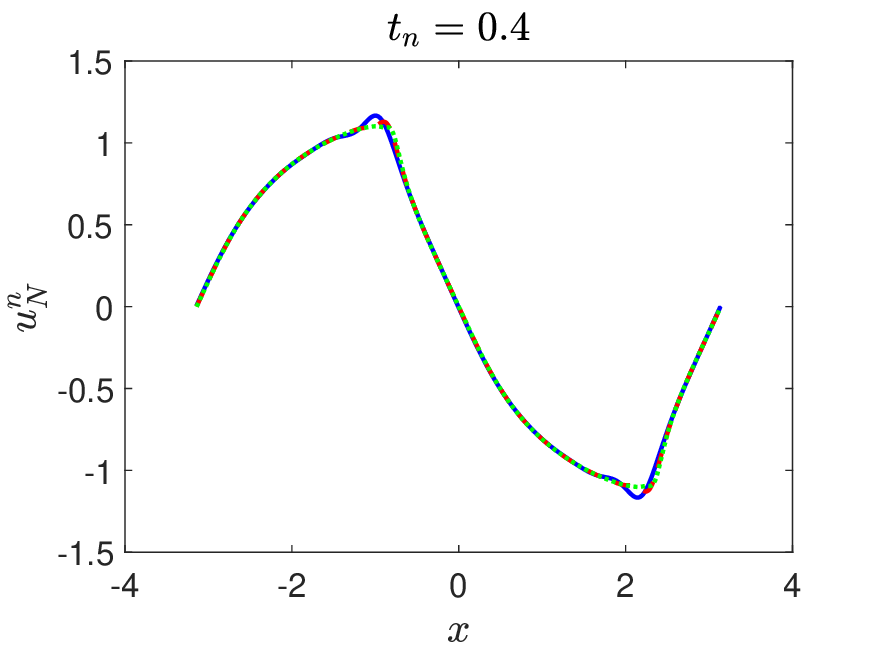}
	\end{subfigure} \\
	\begin{subfigure}[b]{0.32\textwidth}
		\includegraphics[width=\linewidth]{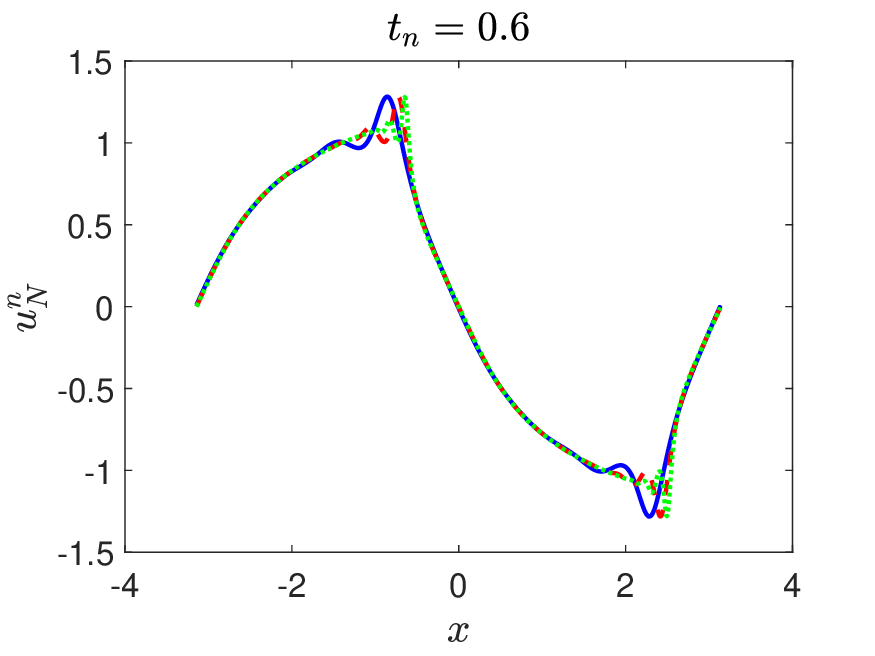}
	\end{subfigure} 
	\begin{subfigure}[b]{0.32\textwidth}
		\includegraphics[width=\linewidth]{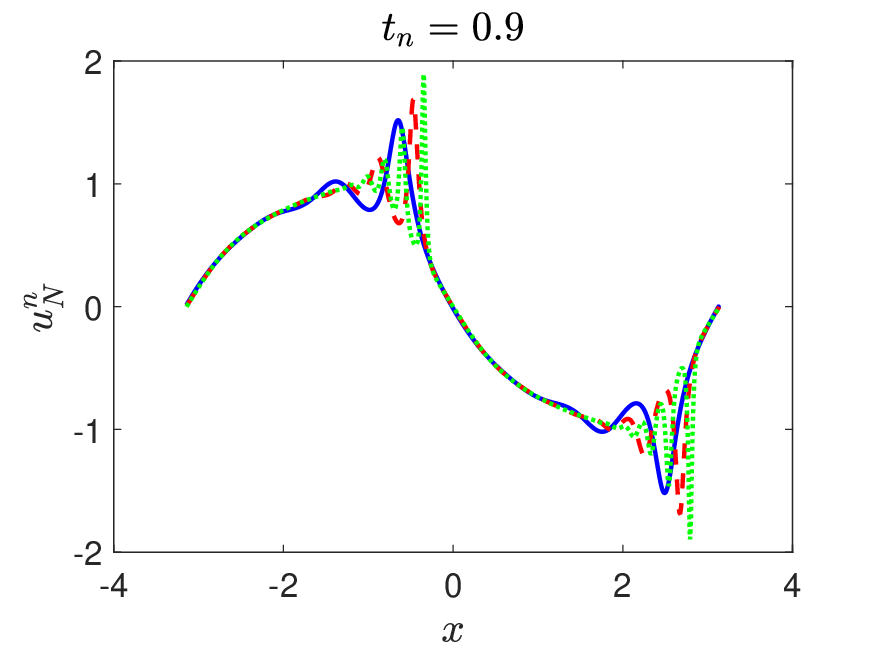}
	\end{subfigure}
	\caption{Snapshots of the numerical solution at different time instants, computed using the CN scheme with $\tau = 10^{-5}$, $N = 1024$, and $\gamma_{0}=2$, for $\varepsilon = 1/10$, $1/20$, and $1/40$.} \label{fig: slice Ex1}
\end{figure}

The initial data is taken as 
\begin{align*}
	u_{0}(x) = 1.1 \sin(x + \pi), \qquad x \in \mathbb{T} = (-\pi, \pi).
\end{align*}
For the fixed-point iteration~\eqref{eq: fixed-point iteration} employed in the CN scheme, the stopping criterion is set as
$
	\bigl\| u_{N,(i+1)}^{n+1} - u_{N,(i)}^{n+1} \bigr\|_{H^{\gamma_{0}}} <
	10^{-12},$ 
with the maximum number of iterations capped at $10^{5}$; if this threshold is exceeded, the iteration is deemed divergent and the computation is terminated.
We introduce the parameter $\gamma_{0}$ to denote the norm index used in the stopping criterion, distinguishing it from the norm index $\gamma$ employed in the error measurement and theoretical analysis.  In the experiments that follow, $\gamma_{0}$ takes different values depending on the computational setting.  For the computation of reference solutions and short-time error plots, we use $\gamma_{0} = 3$, which results in higher accuracy for the iterative solutions.  For long-time simulations extending beyond $t = 0.6$, however, the emergence of high-frequency oscillations associated with DSW formation may cause very large norm values that diverge; accordingly, we set $\gamma_{0} = 2$ in such cases to ensure robust convergence. 
Unless stated otherwise, the reference solution is computed using the CN scheme with $\tau = 10^{-5}$ and $N = 1024$.

Figure~\ref{fig: slice Ex1} displays the numerical solution at five representative time instants for three different values of $\varepsilon$.  The computation is performed with the CN scheme using a fine discretization ($\tau = 10^{-5}$, $N = 1024$, $\gamma_{0} = 2$).  For $t \le 0.3$, the numerical solution evolves smoothly and stably for all values of $\varepsilon$ considered. Starting from $t = 0.4$, however, high-frequency oscillations emerge and grow in amplitude, signaling the onset of DSW formation.  This observation is consistent with the existence of a critical time $t_{c}$ (here approximately $0.3 < t_{c} < 0.4$) beyond which the exact solution develops a gradient catastrophe in the dispersionless limit.

\begin{figure}[t!]
	\centering
	\begin{subfigure}[b]{0.45\textwidth}
		\includegraphics[width=\linewidth]{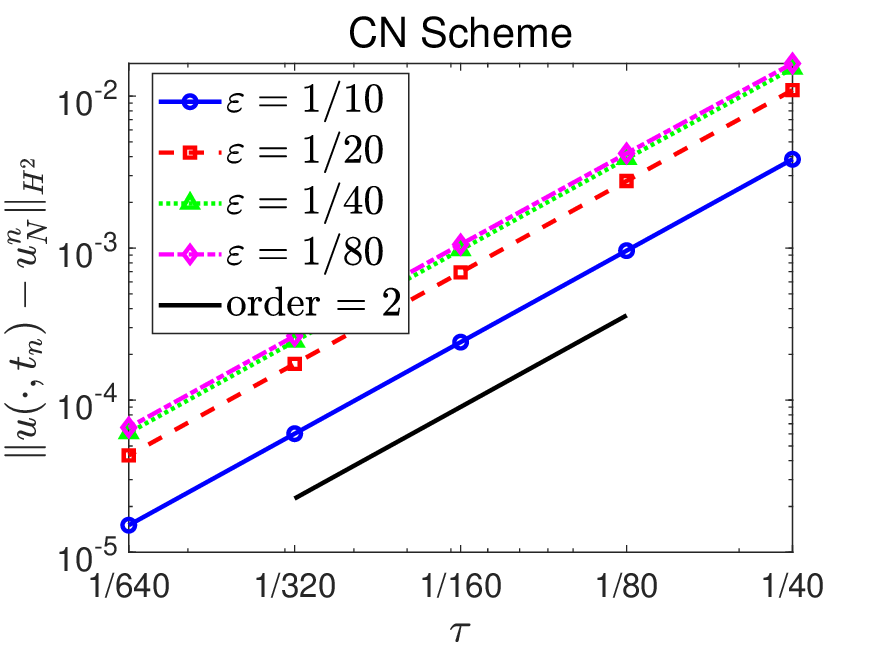}
	\end{subfigure}
	\begin{subfigure}[b]{0.45\textwidth}
		\includegraphics[width=\linewidth]{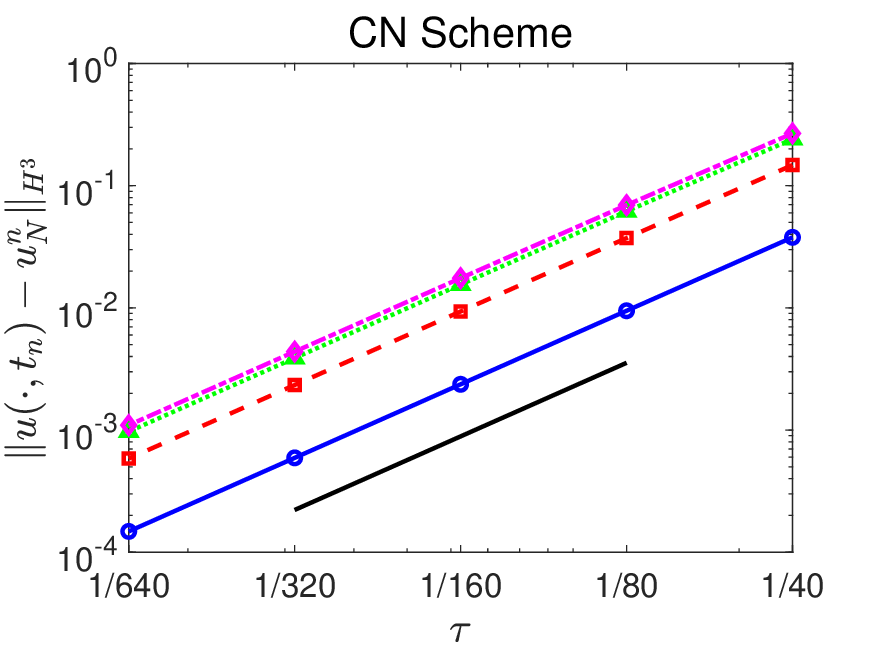}
	\end{subfigure} \\
	\begin{subfigure}[b]{0.45\textwidth}
		\includegraphics[width=\linewidth]{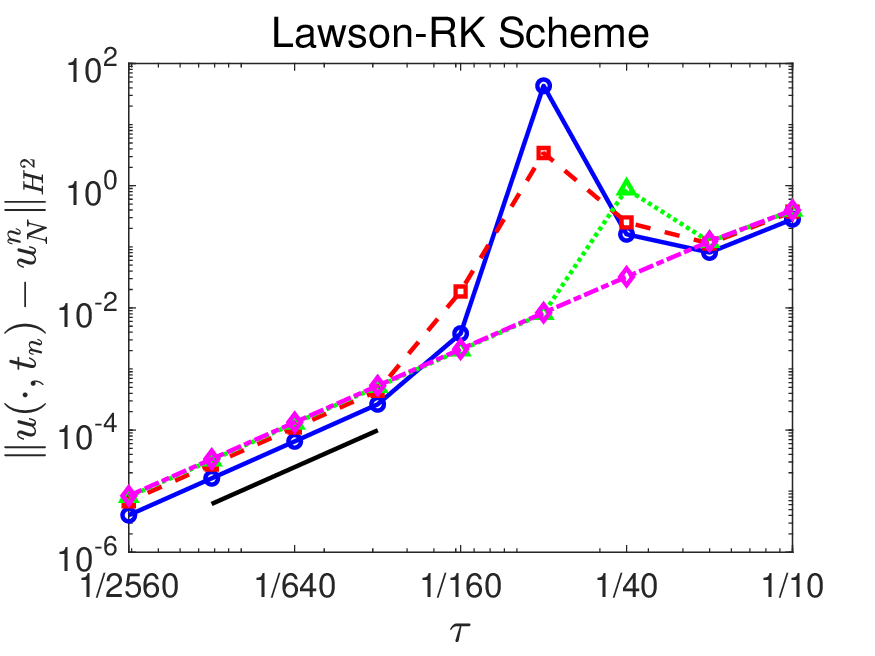}
	\end{subfigure}
	\begin{subfigure}[b]{0.45\textwidth}
		\includegraphics[width=\linewidth]{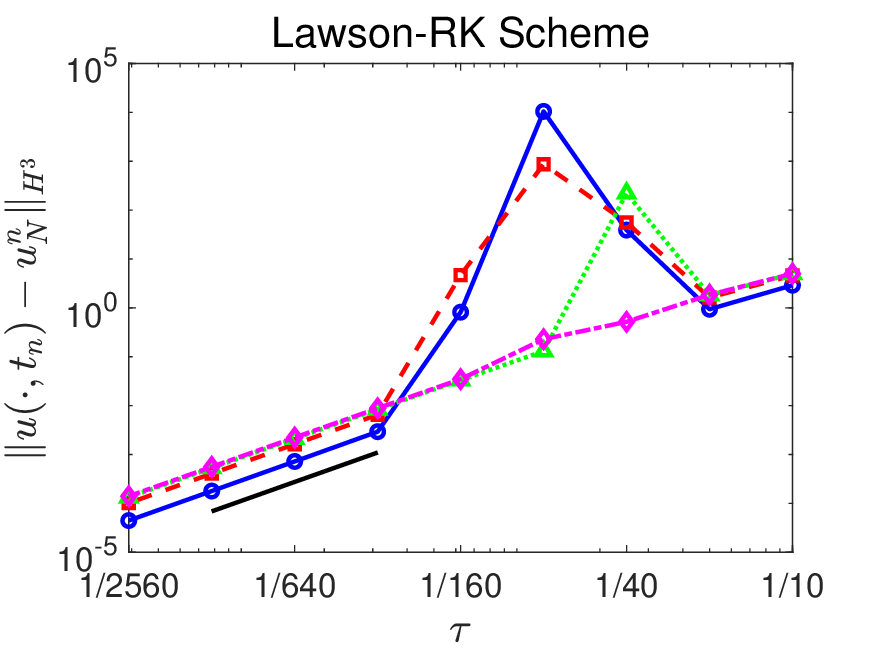}
	\end{subfigure}
	\caption{Temporal convergence: $\| u(\cdot,t_{n}) - u_{N}^{n} \|_{H^{\gamma}}$ at $t_{n} = 0.2$, with $N = 512$ and $\gamma_{0}=3$. Top row: CN scheme; bottom row: Lawson-RK scheme.  The black solid line indicates the reference rate $\mathcal{O}(\tau^{2})$.}
	\label{fig: temporal error Ex1}
\end{figure}

Figure~\ref{fig: temporal error Ex1} presents the temporal discretization errors in the $H^{2}$- and $H^{3}$-norms at $t_{n} = 0.2$, measured against the reference solution, with the spatial resolution fixed at $N = 512$.  For the CN scheme (top row), the error curves align closely with the second-order reference line across all tested values of $\varepsilon$, confirming that the $\mathcal{O}(\tau^{2})$ convergence rate established in Theorem~\ref{thm:CN-convergence} is sharp.  For the Lawson-RK scheme (bottom row), the second-order convergence behavior is observed only when $\tau$ falls below a certain threshold; for larger step sizes, the error deviates from the predicted rate.  Notably, this threshold is considerably less restrictive than the theoretical condition $\tau \lesssim N^{-2}$ required in Theorem~\ref{thm:RK2-convergence}, suggesting that the theoretical step-size restriction for the Lawson-RK scheme may have room for improvement in the analysis.

\begin{figure}[t!]
	\centering
	\begin{subfigure}[b]{0.45\textwidth}
		\includegraphics[width=\linewidth]{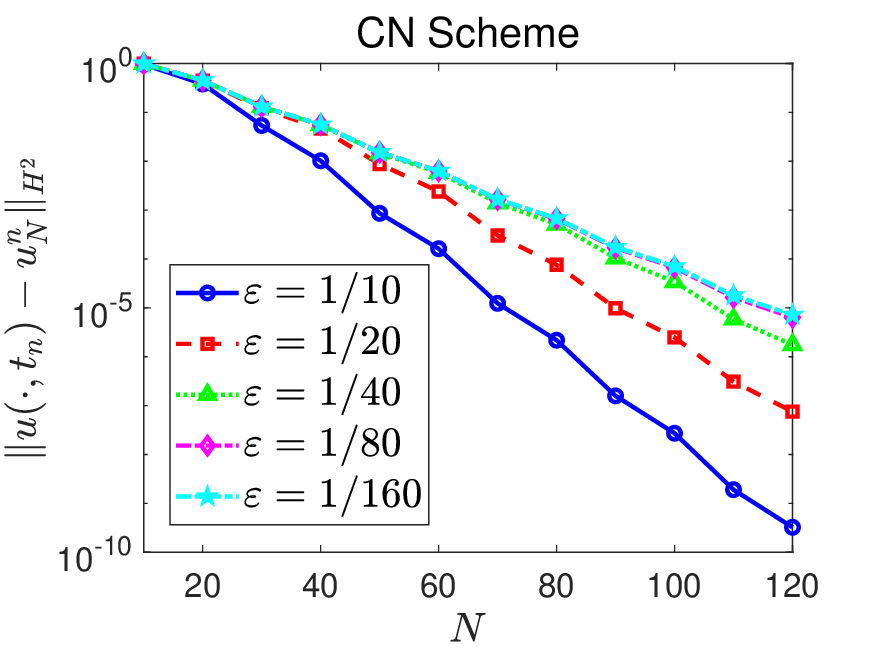}
	\end{subfigure}
	\begin{subfigure}[b]{0.45\textwidth}
		\includegraphics[width=\linewidth]{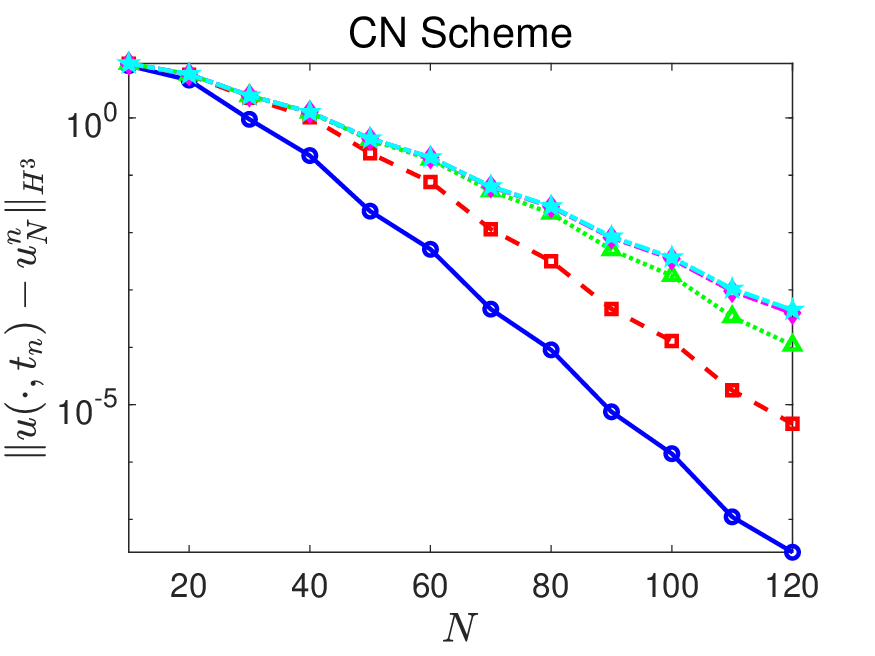}
	\end{subfigure} \\
	\begin{subfigure}[b]{0.45\textwidth}
		\includegraphics[width=\linewidth]{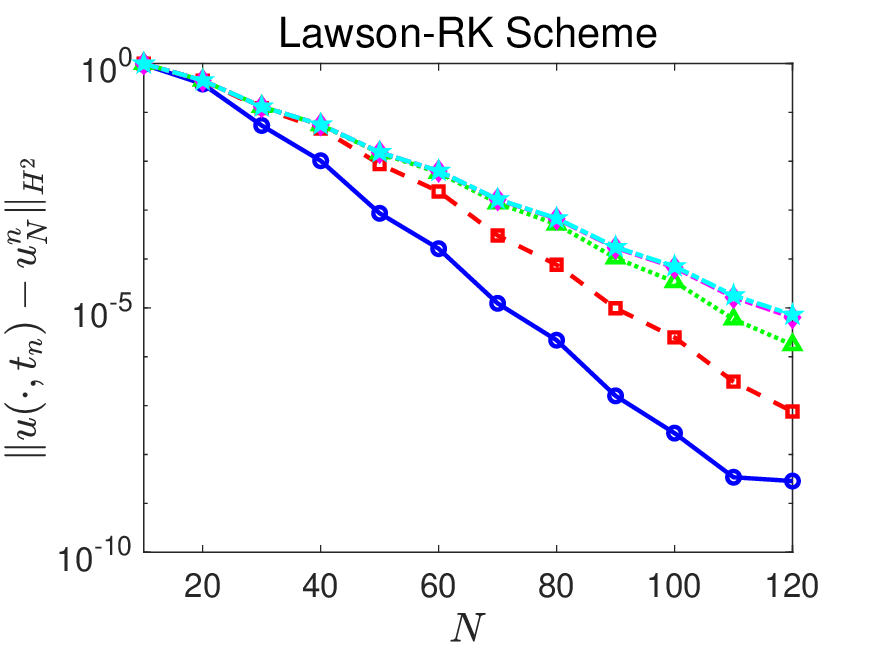}
	\end{subfigure}
	\begin{subfigure}[b]{0.45\textwidth}
		\includegraphics[width=\linewidth]{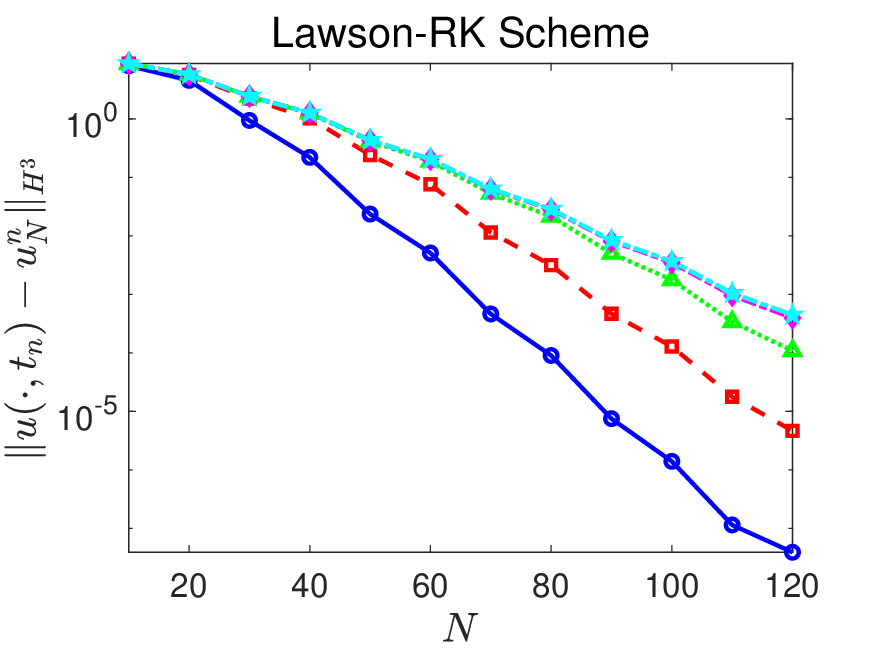}
	\end{subfigure}
	\caption{Spatial convergence: $\| u(\cdot,t_{n}) - u_{N}^{n} \|_{H^{\gamma}}$ at $t_{n} = 0.2$, with $\tau = 10^{-5}$ and $\gamma_{0}=3$.  Top row: CN scheme; bottom row: Lawson-RK scheme.} \label{fig: spatial error Ex1}
\end{figure}

Figure~\ref{fig: spatial error Ex1} shows the spatial discretization errors at $t_{n} = 0.2$ with the time step fixed at $\tau = 10^{-5}$ (so that the temporal error is negligible).  Both schemes exhibit spectral accuracy, as evidenced by the rapid decay of the error with increasing $N$.  A closer inspection reveals that the value of $\varepsilon$ exerts a mild influence on the rate of error decay: for larger values of $\varepsilon$, the error curves descend somewhat more steeply.  We attribute this behavior to the fact that $\varepsilon$ may affect the Sobolev norms of the exact solution, which in turn
modulate the prefactor in the spatial error bound.  Crucially, however, this influence saturates as $\varepsilon$ becomes sufficiently small, and for very small values of $\varepsilon$ the error curves are virtually indistinguishable.  Consequently, the spectral accuracy predicted by our theoretical analysis remains valid uniformly in $\varepsilon$.

\begin{table}[t!]
	\centering
	\caption{Step-size restriction test for the CN scheme: dependence of $\tau$ on $N$, with $\varepsilon = 1/512$, $T_{\max} = 0.5$, and $\gamma_{0}=3$.} \label{tab: iteration test 1 Ex1}
	\fontsize{10pt}{15pt}\selectfont
	\setlength{\tabcolsep}{10pt}
	\begin{tabular}{c|cccccc}
		\hline
		~ &$N = 16$ & $N = 32$ & $N = 64$ & $N = 128$ & $N = 256$ & $N = 512$ \\ \hline
		$\tau = 1/4$ & Y & N ($n = 0$) & N ($n = 0$) & N ($n = 0$) & N ($n = 0$) & N ($n = 0$) \\ 
		$\tau = 1/8$ & Y & N ($n = 0$) & N ($n = 0$) & N ($n = 0$) & N ($n = 0$) & N ($n = 0$) \\
		$\tau = 1/16$ & Y & Y & N ($n = 0$) & N ($n = 0$) & N ($n = 0$) & N ($n = 0$) \\
		$\tau = 1/32$ & Y & Y & Y & N ($n = 0$) & N ($n = 0$) & N ($n = 0$) \\ 
		$\tau = 1/64$ & Y & Y & Y & Y & N ($n = 0$) & N ($n = 0$) \\ 
		$\tau = 1/128$ & Y & Y & Y & Y & Y & N ($n = 0$) \\ \hline
	\end{tabular}
\end{table}

\begin{table}[!ht]
	\centering
	\caption{Step-size restriction test for the CN scheme: dependence of $\tau$ on $\varepsilon$, with $N = 512$, $T_{\max} = 0.5$, and $\gamma_{0}=3$.} \label{tab: iteration test 2 Ex1}
	\fontsize{10pt}{15pt}\selectfont
	\setlength{\tabcolsep}{10pt}
	\begin{tabular}{c|cccccc}
		\hline
		~ &$\varepsilon = 1/16$ & $\varepsilon = 1/32$ & $\varepsilon = 1/64$ & $\varepsilon = 1/128$ & $\varepsilon = 1/256$ & $\varepsilon = 1/512$ \\ \hline
		$\tau = 1/4$ & N ($n = 0$) & N ($n = 0$) & N ($n = 0$) & N ($n = 0$) & N ($n = 0$) & N ($n = 0$) \\ 
		$\tau = 1/8$ & Y & N ($n = 0$) & N ($n = 0$) & N ($n = 0$) & N ($n = 0$) & N ($n = 0$) \\ 
		$\tau = 1/16$ & Y & Y & N ($n = 0$) & N ($n = 0$) & N ($n = 0$) & N ($n = 0$) \\ 
		$\tau = 1/32$ & Y & Y & Y & N ($n = 0$) & N ($n = 0$) & N ($n = 0$) \\ 
		$\tau = 1/64$ & Y & Y & Y & N ($n = 30$) & N ($n = 0$) & N ($n = 0$) \\ 
		$\tau = 1/128$ & Y & Y & Y & Y & N ($n = 0$) & N ($n = 0$) \\ \hline
	\end{tabular}
\end{table}

We next examine the convergence behavior of the fixed-point iteration for the CN scheme.  Tables~\ref{tab: iteration test 1 Ex1} and~\ref{tab: iteration test 2 Ex1} report, respectively, the dependence of the admissible time step $\tau$ on the spatial discretization parameter $N$ (with $\varepsilon = 1/512$ fixed) and on the dispersion parameter $\varepsilon$ (with $N = 512$ fixed). In these tables, ``Y'' indicates that the fixed-point iteration converges within the prescribed maximum number of iterations throughout the entire time interval $[0, T_{\max}]$ with $T_{\max} = 0.5$, whereas ``N($n = n^{*}$)'' indicates that the scheme diverges at time level $n^{*}$, i.e., the iteration exceeds the maximum allowable count.  A clear pattern emerges: the admissible time step $\tau$ is essentially constrained linearly by both $N$ and $\varepsilon$ in a first-order manner, which is in good qualitative agreement with the theoretical step-size condition $\tau \lesssim \min\{ \varepsilon^{1+\sigma_{0}},\; N^{-(1+\sigma_{0})} \}$ for any $\sigma_{0} > 0$ derived in Lemma~\ref{lem:contraction-map}.



\begin{figure}[h!]
	\centering
	\begin{subfigure}[b]{0.32\textwidth}
		\includegraphics[width=\linewidth]{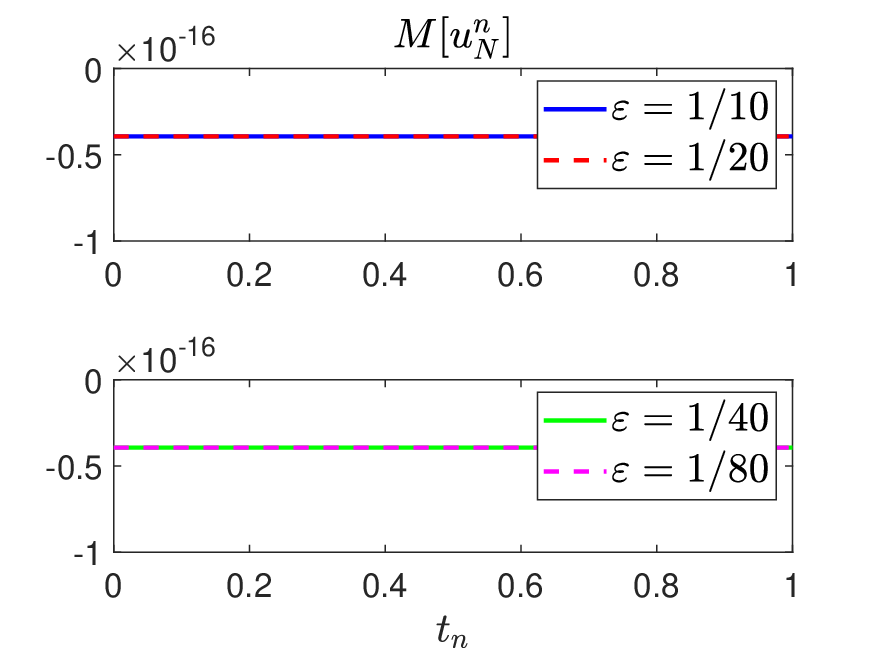}
	\end{subfigure}
	\begin{subfigure}[b]{0.32\textwidth}
		\includegraphics[width=\linewidth]{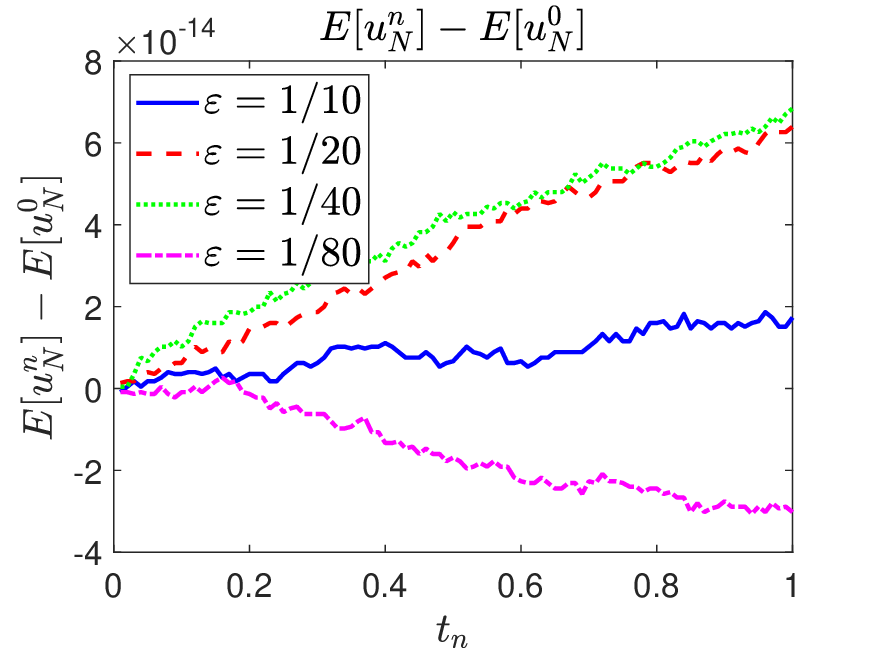}
	\end{subfigure}
	\begin{subfigure}[b]{0.32\textwidth}
		\includegraphics[width=\linewidth]{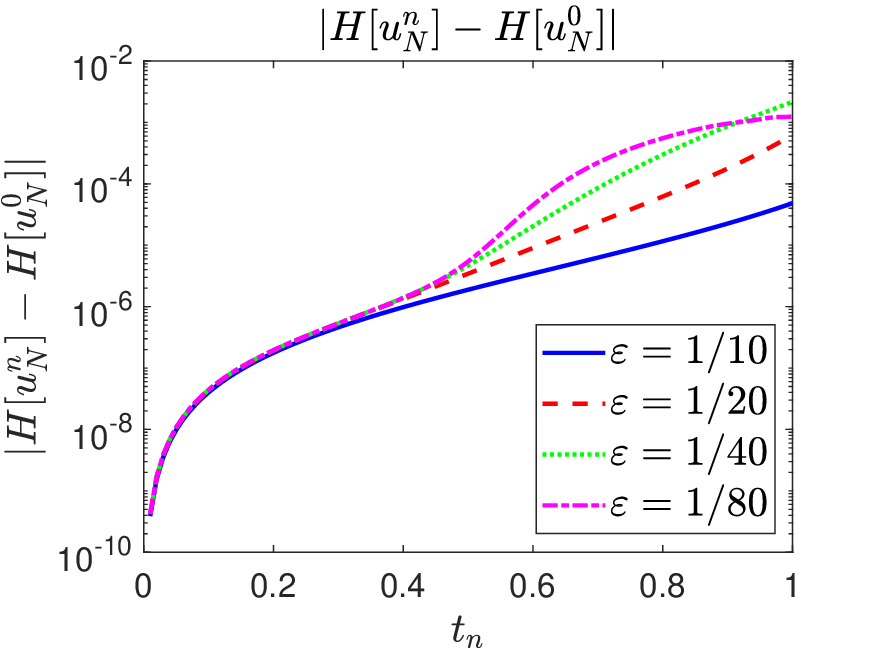}
	\end{subfigure}
	\caption{Evolution of the conserved quantities for the CN scheme, with $\tau = 0.01$, $N = 256$, and $\gamma_{0}=2$.  From left to right: mass $M[u]$, energy $E[u]$, and Hamiltonian $H[u]$.} \label{fig: conservation Ex1}
\end{figure}

\begin{figure}[h!]
	\centering
	\begin{subfigure}[b]{0.32\textwidth}
		\includegraphics[width=\linewidth]{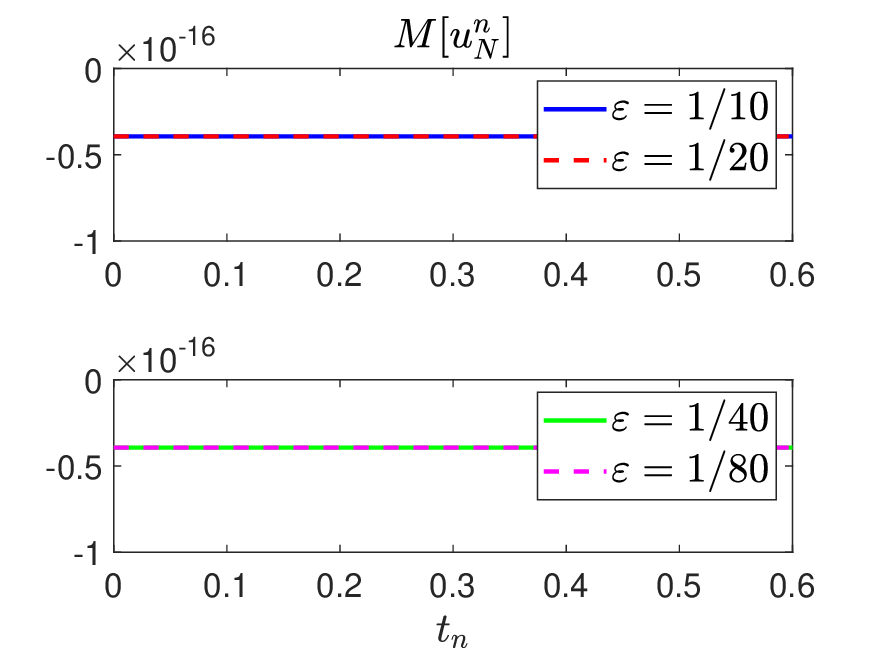}
	\end{subfigure}
	\begin{subfigure}[b]{0.32\textwidth}
		\includegraphics[width=\linewidth]{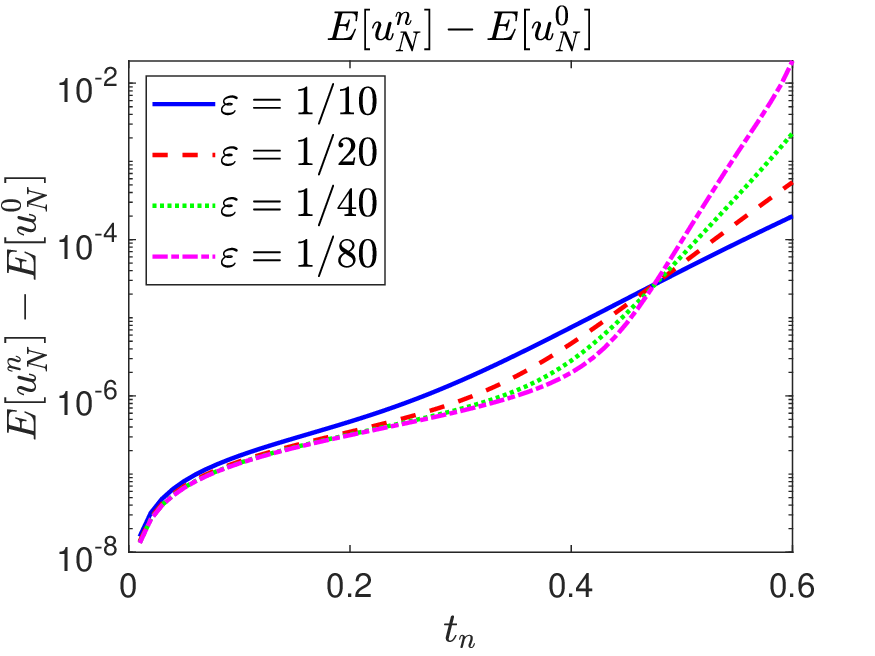}
	\end{subfigure}
	\begin{subfigure}[b]{0.32\textwidth}
		\includegraphics[width=\linewidth]{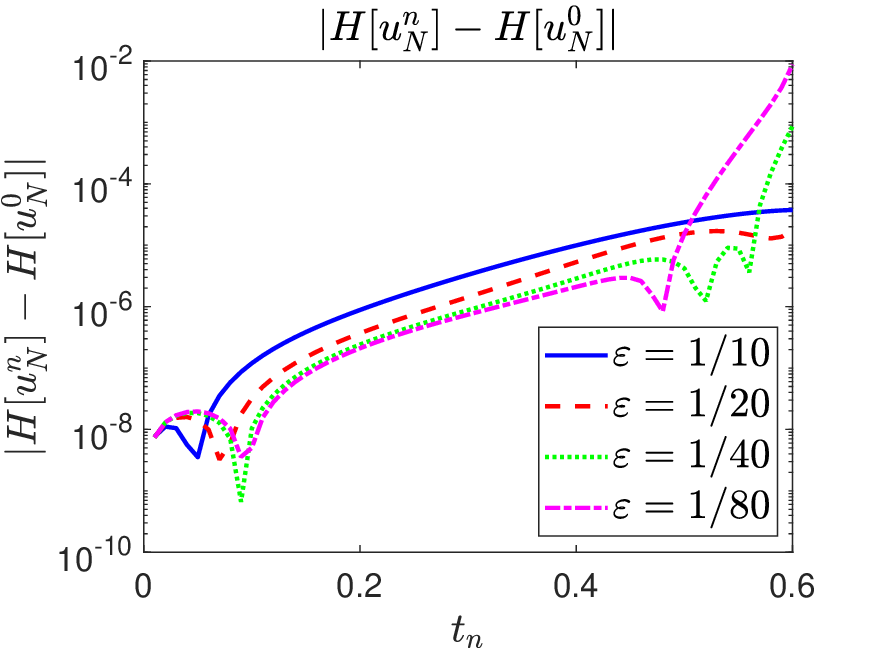}
	\end{subfigure}
	\caption{Evolution of the conserved quantities for the Lawson-RK scheme, with $\tau = 0.01$, $N = 256$, and $\gamma_{0}=2$.  From left to right: mass $M[u]$, energy $E[u]$, and Hamiltonian $H[u]$.} \label{fig: conservation Ex1 for RK2}
\end{figure}

Finally, we examine the preservation of the three invariants---mass $M[u]$, energy $E[u]$, and Hamiltonian $H[u]$---for both schemes, with the discretization parameters fixed at $\tau = 0.01$, $N = 256$, and $\gamma_{0} = 2$.  Figure~\ref{fig: conservation Ex1} shows the results for the CN scheme. As guaranteed by Lemma~\ref{lem: conservative law}, the discrete mass and energy are preserved exactly (up to machine precision) throughout the entire simulation. The numerical value of the Hamiltonian is not preserved. Especially after $t_c$, the drift is more significant, meaning that the performance of the scheme is deteriorating due to the occurrence of multiscale phenomenon. Figure~\ref{fig: conservation Ex1 for RK2} presents the corresponding results for the Lawson-RK scheme.  Here, only the mass is conserved exactly, while both the energy and the Hamiltonian deviate from their initial values, in accordance with the theoretical expectation. In particular, things also deteriorate after $t_c$. 

\begin{figure}[h!]
	\centering
	\begin{subfigure}[b]{0.45\textwidth}
		\includegraphics[width=\linewidth]{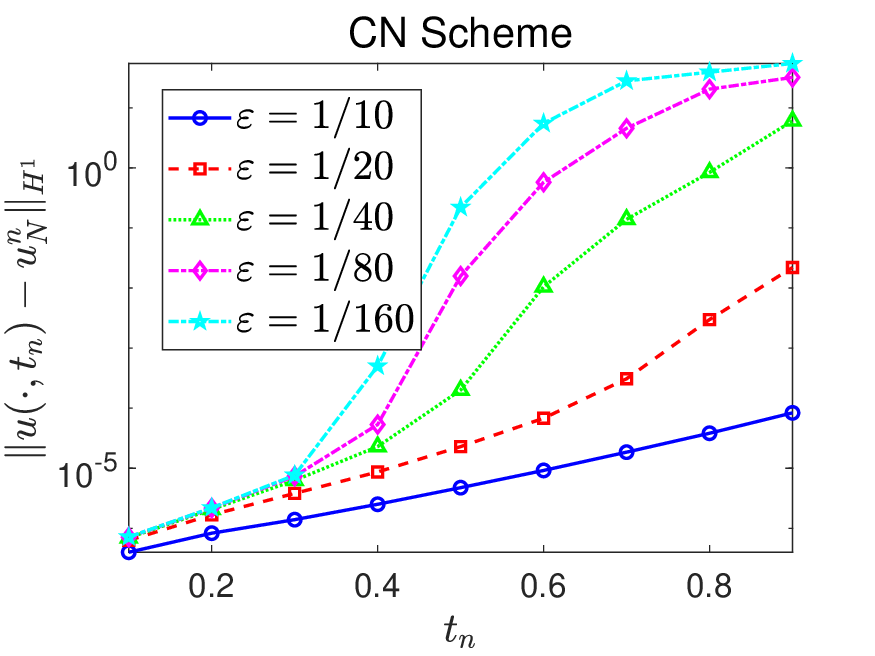}
	\end{subfigure}
	\begin{subfigure}[b]{0.45\textwidth}
		\includegraphics[width=\linewidth]{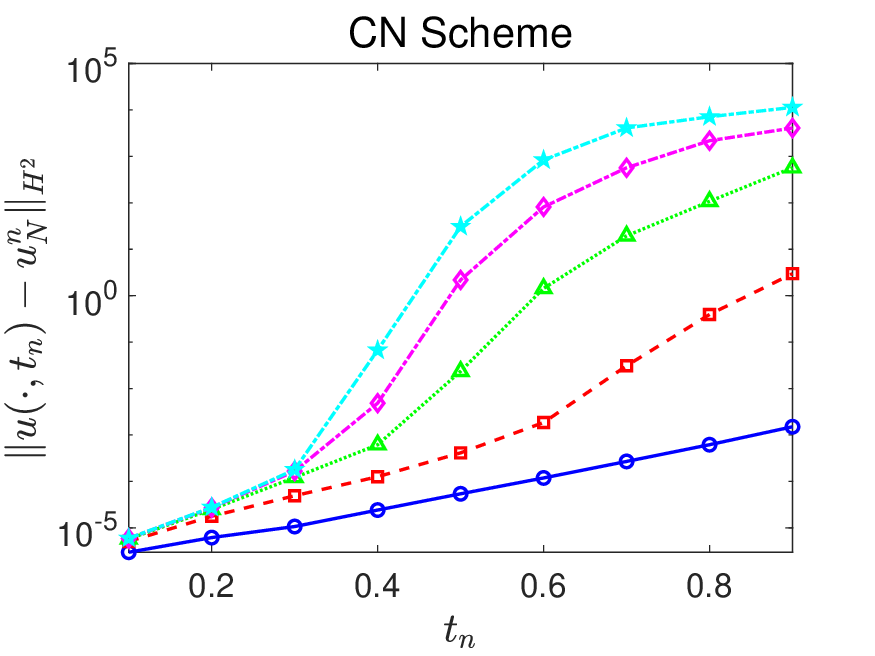}
	\end{subfigure} \\
	\begin{subfigure}[b]{0.45\textwidth}
		\includegraphics[width=\linewidth]{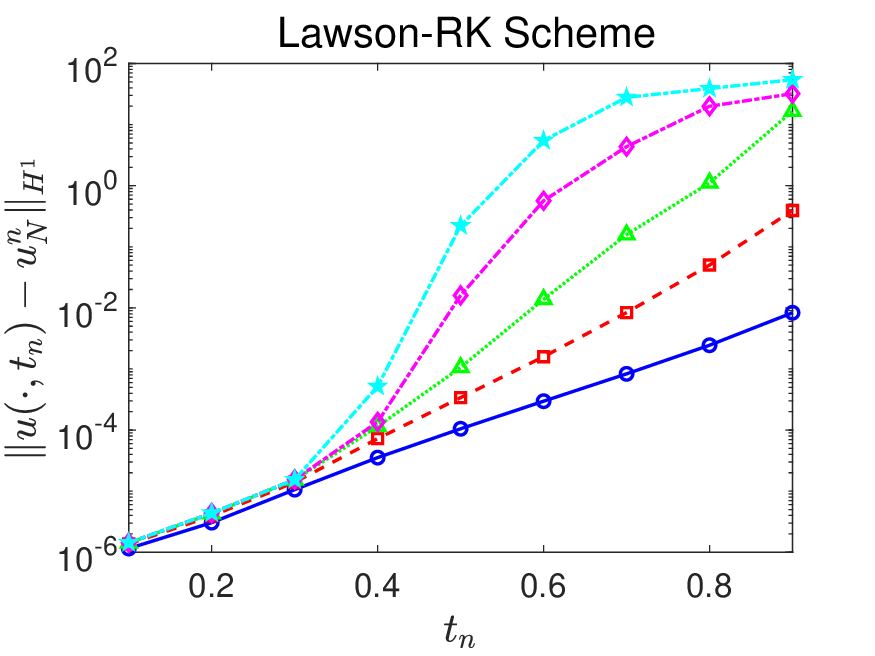}
	\end{subfigure}
	\begin{subfigure}[b]{0.45\textwidth}
		\includegraphics[width=\linewidth]{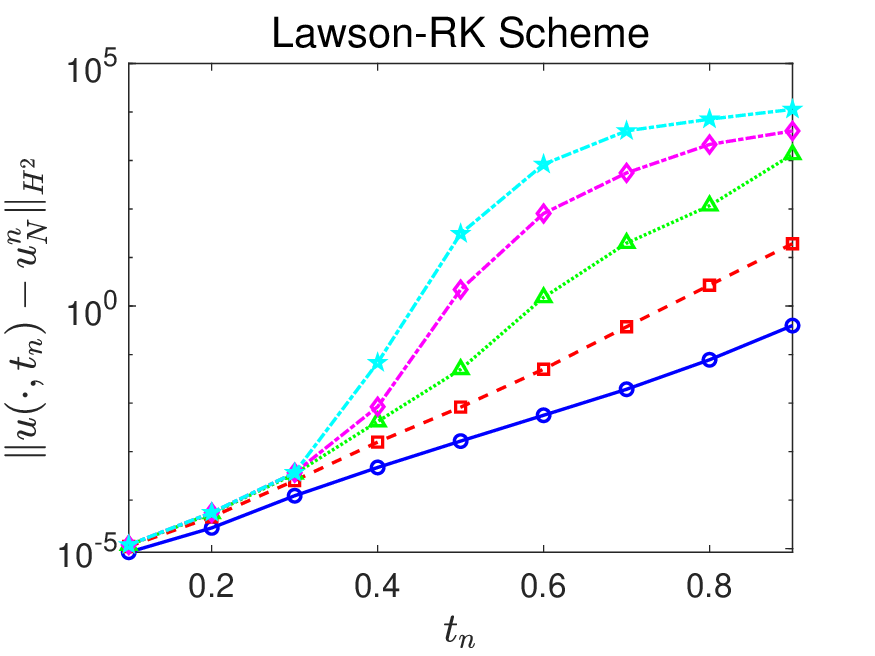}
	\end{subfigure}
	\caption{Error evolution over the full time interval: $\| u(\cdot,t_{n}) - u_{N}^{n} \|_{H^{\gamma}}$ for $t_{n} = 0.1, 0.2, \dots, 0.9$, with $\tau = 0.001$, $N = 256$, and $\gamma_{0}=2$.  Top row: CN scheme; bottom row: Lawson-RK scheme.  The reference solution is the CN numerical solution used $\gamma_{0}=2$.} \label{fig: error evolution Ex1}
\end{figure}

All the convergence tests presented above were conducted at $t_{n} = 0.2$, which lies safely before the critical time $t_{c}$.  To illustrate what happens if one attempts to push the computation beyond $t_{c}$, we plot in Figure~\ref{fig: error evolution Ex1} the $H^{1}$- and $H^{2}$-errors as functions of $t_{n}$ over the extended interval $t_{n} \in [0.1, 0.9]$, for $\varepsilon = 1/10$, $1/20$, $1/40$, $1/80$, and $1/160$.  For $t_{n} \le 0.3$, the error increases only linearly with $t_{n}$ on the logarithmic vertical axis (equivalently, exponentially in time), and the influence of $\varepsilon$ is mild; moreover, this influence saturates as $\varepsilon \to 0$, consistent with the $\varepsilon$-uniform estimates established in Theorems~\ref{thm:CN-convergence} and~\ref{thm:RK2-convergence}. Starting from $t_{n} = 0.4$, however, a dramatic change occurs: as $\varepsilon$ decreases, the error grows substantially, indicating that the numerical solution can no longer reliably approximate the exact solution in the post-$t_{c}$ regime under the current discretization strategy.  This observation explains why our convergence analysis is restricted to times $T < t_{c}$: beyond the critical time, the solution of \eqref{eq: gKdV} becomes highly oscillatory in both time and space with oscillating frequencies inversely dependent on $\varepsilon$, making classical numerical approximations inaccurate and  inefficient. This, on the other hand, calls for developments of multiscale techniques and analysis that would be the direction of our future work.


\section{Conclusion}\label{sec:con} 
This work considers numerical analysis for the generalized Korteweg-de Vries (gKdV) equation in the dispersionless limit regime before the formation of dispersive shock waves (DSW). Two classical schemes---the Crank-Nicolson (CN) scheme and a Lawson-type Runge-Kutta (Lawson-RK) scheme, both combined with the Fourier pseudo-spectral method in space---are studied. 
The Lawson-RK scheme, offering an explicit time-stepping approach,  preserves discrete mass, while the implicit CN scheme conserves both discrete mass and energy. Our primary contribution lies in the rigorous proof that both of these fully discrete schemes achieve optimal second-order temporal accuracy and spectral spatial accuracy. Crucially, the derived error constants are independent of the dispersion parameter $\varepsilon$, ensuring reliable convergence as $\varepsilon$ approaches zero. 
For the CN scheme, we establish its unique solvability at each time step, facilitated by a fixed-point iteration algorithm whose convergence is uniform with respect to $\varepsilon$, subject to a specific step-size restriction. Numerical experiments presented herein serve to validate our theoretical findings. These experiments confirm the sharp temporal and spatial convergence rates with the $\varepsilon$-uniform accuracy and support the necessity of the theoretical step-size restriction on stability of the fixed-point iteration for the CN scheme. 
Computations beyond the critical time for the occurrence of DSW are also explored, where the performance of the schemes quickly deteriorates. 
 Our findings in total 
give credit to classical methods before DSW and leave the demand for  sophisticated multiscale numerical techniques afterwards,  
 forming an important step towards understanding and reliably simulating the gKdV equation in the dispersionless limit regime.

\section*{Acknowledgment}
T. Wang and X. Zhao are supported by National Key Research and Development Program of China, National MCF Energy R\&D Program (No. 2024YFE03240400), NSFC 42450275, 12271413. B. Li is supported by the National Natural Science Foundation of China (Grant No. 12401308) and the Natural Science Foundation of Sichuan Province (Grant No. 2026NSFSC0749). 

\bibliographystyle{siam}
\bibliography{ref}

\end{document}